\documentclass[12pt]{amsart}     

\usepackage{amsmath}
\usepackage{amsthm}
\usepackage{amsfonts}
\usepackage{latexsym}
\usepackage{float}
\usepackage{upref}
\restylefloat{figure}
\usepackage{mathptmx}
\usepackage[dvips]{graphicx}

\usepackage{mathptmx}
\usepackage{times}

\newtheorem{theorem}{Theorem}[section]
\newtheorem{lem}[theorem]{Lemma}
\newtheorem{cor}[theorem]{Corollary}

\theoremstyle{definition}
\newtheorem{defn}[theorem]{Definition}
\theoremstyle{remark}
\newtheorem{rem}[theorem]{Remark}

\newtheorem{theorem*}{Theorem}
\newtheorem{assumption*}[theorem*]{Assumption}

\usepackage{latexsym} 
\usepackage{amsmath}
\usepackage{amsthm}
\usepackage{amsfonts}
\usepackage{amssymb}
\usepackage{amsxtra}
\usepackage{amscd}
\usepackage{ifthen}
\usepackage{graphicx}
\usepackage{enumerate}
\usepackage[hidelinks, colorlinks, allcolors = black, breaklinks = true, pagebackref=false]{hyperref}
\renewcommand{\geq}{\geqslant}
\renewcommand{\ge}{\geqslant}

\renewcommand{\le}{\leqslant}

\usepackage[usenames,dvipsnames,svgnames,table]{xcolor}

\usepackage[utf8]{inputenc}    
\usepackage[T1]{fontenc}       

\usepackage[margin=1in]{geometry}

\hypersetup{
pdfauthor={Benjamin Gess, Jonas Sauer, Eitan Tadmor},
pdftitle={Optimal regularity in time and space for the porous medium equation},
breaklinks=true,
colorlinks=true,
linkcolor=blue,
citecolor=blue,
urlcolor=blue,
filecolor=blue,
}

\numberwithin{equation}{section} 

\newcommand{\eps}{\varepsilon}
\newcommand{\dd}{\,\mathrm{d}}

\newcommand{\ti}{\tilde}
\newcommand{\td}{\tilde}
\newcommand{\vp}{\varphi}

\renewcommand{\div}{\mathrm{div}\,}     

\newcommand{\uf}{u}
\newcommand{\ubb}{u_{BB}}

\newcommand{\vf}{v}

\newcommand{\f}{f}
\newcommand{\g}{g}

\renewcommand{\t}{\theta}
\newcommand{\vt}{\vartheta}
\newcommand{\tauz}{z}

\newenvironment{pdeq}{ \left\{ \begin{aligned}}{\end{aligned}\right.}

\newcommand{\vpp}[1]{\ensuremath{\left(#1\right)}}
\newcommand{\cald}{{\mathcal D}}

\newcommand{\calf}{{\mathcal F}}

\newcommand{\call}{{\mathcal L}}
\newcommand{\mcL}{{\mathcal L}}
\newcommand{\calm}{{\mathcal M}}

\newcommand{\calo}{{\mathcal O}}

\newcommand{\cals}{{\mathcal S}}

\newcommand{\calz}{{\mathcal Z}}

\newcommand{\R}{\mathbb{R}}
\newcommand{\Z}{\mathbb{Z}}
\newcommand{\C}{\mathbb{C}}

\newcommand{\N}{\mathbb{N}}

\DeclareMathOperator{\id}{\textsf{id}}

\DeclareMathOperator{\supp}{supp}

\DeclareMathOperator{\sgn}{sgn}

\newcommand{\diffop}{\call}

\newcommand{\Fdiffop}{\diffop(i\tau,i\xi,\vf)}

\newcommand{\set}[1]{\ensuremath{\{#1\}}}

\newcommand{\setc}[2]{\ensuremath{\{#1\ \lvert\ #2\}}}

\newcommand{\rhon}{\beta}
\newcommand{\rhoe}{\rho}

\renewcommand{\d}{\delta}
\renewcommand{\b}{\beta}
\newcommand{\s}{\sigma}

\newcommand{\mcF}{\calf}

\newcommand{\norm}[1]{\lVert#1\rVert}

\newcommand{\opnorm}[1]{{\lvert\kern-0.25ex\lvert\kern-0.25ex\lvert #1 \rvert\kern-0.25ex\rvert\kern-0.25ex\rvert}}

\newcommand{\WSR}[2]{W^{#1,#2}}

\newcommand{\DSR}[2]{\dot{W}^{#1,#2}}

\newcommand{\CR}[1]{C^{#1}}  
\newcommand{\LR}[1]{L^{#1}}

\newcommand{\CRi}{\CR \infty}
\newcommand{\CRci}{\CR \infty_c}

\newcommand{\tin}{\text{in }}
\newcommand{\tif}{\text{if }}
\newcommand{\ton}{\text{on }}

\newcommand{\tand}{\text{and }}

\newcommand{\newCCtr}[2][d]{
\newcounter{#2}\setcounter{#2}{0}
\expandafter\xdef\csname kyedtheconst#2\endcsname{#1}
}
\newcommand{\Cc}[2][nolabel]{
\stepcounter{#2}
\expandafter\ensuremath{\csname kyedtheconst#2\endcsname_{\arabic{#2}}}
\ifthenelse{\equal{#1}{nolabel}}
{}
{\expandafter\xdef\csname kyedconst#1\endcsname
{\expandafter\ensuremath{\csname kyedtheconst#2\endcsname_{\arabic{#2}}}}}
}
\newcommand{\Ccn}[2][nolabel]{
\expandafter\ensuremath{\csname kyedtheconst#2\endcsname}
\ifthenelse{\equal{#1}{nolabel}}
{}
{\expandafter\xdef\csname kyedconst#1\endcsname
{\expandafter\ensuremath{\csname kyedtheconst#2\endcsname}}}
}
\newcommand{\CcSetCtr}[2]{
\setcounter{#1}{#2}
}
\newcommand{\Cclast}[1]{
\expandafter\ensuremath{\csname kyedtheconst#1\endcsname_{\arabic{#1}}}
}

\newcommand{\Ccllast}[1]{
\addtocounter{#1}{-1}
\expandafter\ensuremath{\csname kyedtheconst#1\endcsname_{\arabic{#1}}}
\addtocounter{#1}{1}
}
\newcommand{\const}[1]{
\expandafter{\ifcsname kyedconst#1\endcsname
  \csname kyedconst#1\endcsname
\else
  \errmessage{Undefined Kyedconstant #1.}%
\fi}
}

\makeatletter
\def\blfootnote{\xdef\@thefnmark{}\@footnotetext}
\makeatother

\theoremstyle{plain}
\newtheorem{thm}{Theorem}[section]
\newtheorem*{thm*}{Theorem}
\theoremstyle{remark}

\newCCtr[c]{c}
\newCCtr[C]{C}
\newCCtr[M]{M}
\newCCtr[B]{B}
\newCCtr[\epsilon]{eps}
\CcSetCtr{eps}{-1}

\begin{document}                        


\title[Optimal regularity for the porous medium equation]{Optimal regularity in time and space\\for the porous medium equation}


\author{Benjamin~Gess}
\address{Benjamin~Gess\newline
Max Planck Institute for Mathematics in the Sciences\newline Inselstr.~22, 04103 Leipzig, Germany\newline
and Fakult\"at f\"ur Mathematik, Universit\"at Bielefeld\newline Universit\"atstr. 25, 33615 Bielefeld, Germany }
\email{bgess@mis.mpg.de}
\author{Jonas~Sauer}
\address{Jonas~Sauer\newline
Max Planck Institute for Mathematics in the Sciences\newline Inselstr.~22, 04103 Leipzig, Germany}
\email{Jonas.Sauer@mis.mpg.de}
\author{Eitan Tadmor}
\address{Eitan Tadmor\newline
Department of Mathematics, Institute for Physical Sci. \& Technology\newline
and Center for Scientific Computation and Mathematical Modeling\newline
University of Maryland, College Park}
\email{tadmor@cscamm.umd.edu}


\begin{abstract}
Regularity estimates in time and space for solutions to the porous medium equation are shown in the scale of Sobolev spaces. In addition, higher spatial regularity for powers of the solutions is obtained. Scaling arguments indicate that these estimates are optimal. In the linear limit, the proven regularity estimates are consistent with the optimal regularity of the linear case.
\end{abstract}


\date{\today}

\maketitle   


\setcounter{tocdepth}{1}
\tableofcontents



\section{Introduction}\label{Int}

We prove estimates on the time and space regularity of solutions to porous medium equations
 \begin{align}\label{pme_sys}
 \begin{pdeq}
  \partial_t\uf-\Delta\uf^{[m]} &= S && \tin (0,T)\times \R^d,\\
  \uf(0)&=\uf_0 && \tin \R^d
 \end{pdeq}
 \end{align} 
 where  $u^{[m]}:=|u|^{m-1} u$ with $m>1$, $u_0 \in L^1(\R^d)$ and $S\in L^1((0,T)\times \R^d)$.  
  Solutions to porous medium equations are known to exhibit nonlinear phenomena like slow diffusion or filling up of holes at finite rate: if the initial data is compactly supported, then the support of the solution evolves with a free boundary that has finite speed of propagation. The solution close to the boundary is not smooth even for smooth initial data and zero forcing. 
 
 Despite many works on the problem of regularity of solutions to porous medium equations, until recently established regularity results in the literature in terms of Hölder or Sobolev spaces were restricted to spatial differentiability of order less than one (cf.\@ Ebmeyer \cite{Ebm05}, Tadmor and Tao \cite{TaT07}). For $m\searrow1$ this is in stark contrast to the limiting case $m=1$, where $u$ is up to twice weakly differentiable in space. Very recently, the first author has proven optimal spatial regularity for \eqref{pme_sys} in \cite{Ges17} for initial data $u_{0}\in(L^{1}\cap L^{1+\varepsilon})(\R^{d})$ for some $\varepsilon>0$. This leaves open three main aspects addressed in the present work: First, the derivation of optimal\footnote{Optimality is indicated by scaling arguments in Section \ref{Scale} below, and the derived estimates are consistent with the optimal space-time regularity in the linear case $m=1$.} space-time regularity. Second, the limit case $u_{0}\in L^{1}(\R^{d})$, which is of particular importance since it covers the case of the Barenblatt solution for which the estimates are shown to be optimal, cf.\@ Section~\ref{Scale} below. Third, higher order integrability. The solution of these three open problems is the purpose of the present paper.

 The first main result provides optimal space-time regularity for $L^1$ data.
 \begin{thm}\label{cor:pme_l1}
  Let $u_0\in\LR{1}(\R^d)$, $S\in L^1((0,T)\times \R^d)$ and $m\in(1,\infty)$. Let $u$ be the unique entropy solution to \eqref{pme_sys} on $[0,T]\times \R^d$.
 \begin{enumerate}
  \item Let $p\in (1,m]$ and define
 \begin{align*}
  \kappa_t:= \frac{m-p}{p}\frac{1}{m-1}, \quad
  \kappa_x:= \frac{p-1}{p}\frac{2}{m-1}.
 \end{align*}
  Then, for all $\sigma_t\in [0,\kappa_t) \cup \set{0}$ and $\sigma_x\in [0,\kappa_x)$ we have
 \begin{align*}
  u\in \WSR{\sigma_t}{p}(0,T;\WSR{\sigma_x}{p}(\R^d)).
 \end{align*}
  Moreover, we have the estimate
 \begin{align}\label{cor:pme_est1_l1}
  \|u\|_{\WSR{\sigma_t}{p}(0,T;\WSR{\sigma_x}{p}(\R^d))}\lesssim \|u_0\|_{L^1_{x}}^{m} +\|S\|_{L^1_{t,x}}^m +1.
 \end{align}
  \item Suppose $\calo\subset\subset \R^d$. Let $s\in [0,1]$ and define
 \begin{align*}
  p&:=s(m-1)+1, \quad
  \kappa_t:= \frac{1-s}{s(m-1)+1}, \quad
  \kappa_x:= \frac{2s}{s(m-1)+1}.
 \end{align*}
  Then for all $\sigma_t\in [0,\kappa_t)\cup \set{0}$, $\sigma_x\in [0,\kappa_x)\cup\set{0}$ and $q\in [1,p]$ we have
 \begin{align*}
  u\in \WSR{\sigma_t}{q}(0,T;\WSR{\sigma_x}{q}(\calo)).
 \end{align*}
  Moreover, we have the estimate
 \begin{align}\label{cor:pme_est2_l1}
  \|u\|_{\WSR{\sigma_t}{q}(0,T;\WSR{\sigma_x}{q}(\calo))}\lesssim \|u_0\|_{L^1_x}^{m} +\|S\|_{L^1_{t,x}}^m +1.
 \end{align}
 \end{enumerate}
 \end{thm}
 
 In the previous works by Tadmor and Tao \cite{TaT07} and Ebmeyer \cite{Ebm05}, initial data in $L^1 \cap L^\infty$ has been considered. However, the methods employed in these works did not allow a systematic analysis of the order of integrability of the solutions. For example, the results of \cite{Ebm05} are restricted to the particular order of integrability $p=\frac{2}{m+1}$, while \cite{TaT07} is restricted to $p=1$. In the second main result we provide a systematic treatment of higher order integrability. In particular, this includes and generalizes the corresponding results of \cite{Ebm05} in terms of regularity in Sobolev spaces. 
 
 Noting that the regularity of $\uf^{[m]}$ contains information on the time regularity of $\uf$ in virtue of the equation \eqref{pme_sys}, in addition, we analyze the spatial regularity of powers of the solution $u^{\mu}$ for $\mu\in[1,m]$.
  \begin{thm}\label{lem:pme}
 Let $u_0\in L^1(\R^d)\cap L^{\rhoe}(\R^d)$, $S\in L^1([0,T]\times\R^d)\cap L^{\rhoe}([0,T]\times\R^d)$ for some $\rhoe\in(1,\infty)$ and assume $m\in\vpp{1,\infty}$. Let $u$ be the unique entropy solution to \eqref{pme_sys} on $[0,T]\times\R^d$.
 \begin{enumerate}
  \item Let $\mu\in[1,m]$. Then, for all $p\in (1,\frac{m-1+\rhoe}{\mu})$ and $\sigma_x\in[0,\frac{\mu p -1}{p}\frac{2}{m-2+\rhoe})$ we have
 \begin{align*}
  u^{[\mu]}\in \LR{p}(0,T;\WSR{\sigma_x}{p}(\R^d)),
 \end{align*}
 and we have the estimate
 \begin{align}\label{lem:pme_est2}
  \|u^{[\mu]}\|_{\LR{p}(0,T;\WSR{\sigma_x}{p}(\R^d))}\lesssim \|u_0\|_{L^1_x\cap L^{\rhoe}_{x}}^{\mu\rhoe} + \|S\|_{L^1_{t,x}\cap L^{\rhoe}_{t,x}}^{\mu\rhoe}+1.
 \end{align}
  \item Let $p\in (\rhoe,m-1+\rhoe)$ and define
 \begin{align*}
  \kappa_t:= \frac{m-1+\rhoe-p}{p}\frac{1}{m-1}, \quad
  \kappa_x:= \frac{p-\rhoe}{p}\frac{2}{m-1}.
 \end{align*}
  Then, for all $\sigma_t\in [0,\kappa_t)$ and $\sigma_x\in [0,\kappa_x)$ we have
 \begin{align*}
  u\in \WSR{\sigma_t}{p}(0,T;\WSR{\sigma_x}{p}(\R^d)).
 \end{align*}
  Moreover, we have the estimate
 \begin{align}\label{lem:pme_est1}
  \|u\|_{\WSR{\sigma_t}{p}(0,T;\WSR{\sigma_x}{p}(\R^d))}\lesssim \|u_0\|_{L^1_x\cap L^{\rhoe}_{x}}^{\rhoe} + \|S\|_{L^1_{t,x}\cap L^{\rhoe}_{t,x}}^{\rhoe}+1.
 \end{align}
 \end{enumerate}
 \end{thm}
 
  Similarly as in Theorem \ref{cor:pme_l1}, if one restricts to estimates that are localized in space, the rigid interdependency of the coefficients in Theorem \ref{lem:pme} can be relaxed.
  
  \begin{cor}\label{cor:pme}
 Under the assumptions of Theorem \ref{lem:pme}, suppose $\calo\subset\subset \R^d$. 
 \begin{enumerate}
  \item Let $\mu\in[1,m]$. Then, for all $\sigma_x\in[0,\frac{2\mu}{m})$ and $q\in[1,\frac{m}{\mu}]$ we have
 \begin{align*}
  u^{[\mu]}\in \LR{q}(0,T;\WSR{\sigma_x}{q}(\calo)),
 \end{align*}
 and we have the estimate
 \begin{align}\label{cor:pme_est2}
  \|u^{[\mu]}\|_{\LR{q}(0,T;\WSR{\sigma_x}{q}(\calo))}\lesssim \|u_0\|_{L^1_x\cap L^{\rhoe}_{x}}^{\mu\rhoe} + \|S\|_{L^1_{t,x}\cap L^{\rhoe}_{t,x}}^{\mu\rhoe}+1.
 \end{align}
  \item Let $s\in [0,1]$ and define
 \begin{align*}
  p&:=s(m-1)+1, \quad
  \kappa_t:= \frac{1-s}{s(m-1)+1}, \quad
  \kappa_x:= \frac{2s}{s(m-1)+1}.
 \end{align*}
  Then for all $\sigma_t\in [0,\kappa_t)\cup \set{0}$, $\sigma_x\in [0,\kappa_x)\cup \set{0}$ and $q\in [1,p]$ we have
 \begin{align*}
  u\in \WSR{\sigma_t}{q}(0,T;\WSR{\sigma_x}{q}(\calo)).
 \end{align*}
  Moreover, we have the estimate
 \begin{align}\label{cor:pme_est1}
  \|u\|_{\WSR{\sigma_t}{q}(0,T;\WSR{\sigma_x}{q}(\calo))}\lesssim \|u_0\|_{L^1_x\cap L^{\rhoe}_{x}}^{\rhoe} + \|S\|_{L^1_{t,x}\cap L^{\rhoe}_{t,x}}^{\rhoe}+1.
 \end{align}
 \end{enumerate}
 \end{cor}

The methods employed in this work are inspired by Tadmor and Tao \cite{TaT07} and rely on the kinetic form of \eqref{pme_sys}, that is, with $\f(t,x,\vf):=1_{\vf<\uf(t,x)}-1_{\vf<0}$, 
\begin{align}
\partial_{t}\f-m|\vf|^{m-1}\Delta_{x}\f & =\partial_{\vf}q+S(t,x)\delta_{u(t,x)}(\vf),\label{pme_sys_kin}
\end{align}
for a non-negative measure $q$, which allows the usage of averaging lemmata and real interpolation. An essential difference to purely spatial regularity consists in the necessity to work with anisotropic fractional Sobolev spaces, which only in their homogeneous form are nicely adapted to the Fourier analytic methods of this work, much in contrast to the purely spatial case in \cite{Ges17}. This leads to the so-called dominating mixed anisotropic Besov spaces introduced by Schmeisser and Triebel \cite{ScT87}. Passing from these homogeneous anisotropic spaces to standard inhomogeneous fractional Sobolev spaces is delicate and treated in detail below. A main ingredient in the proof of optimal regularity in \cite{Ges17} was the existence of singular moments $\int_{t,x,v}|v|^{-\gamma}q$ for $\gamma\in(0,1)$. This ceases to be true for general $L^{1}$-initial data. This difficulty is overcome in the present work by treating separately the degeneracy at $|v|=0$ and the singularity at $|v|=\infty$ as they appear in \eqref{pme_sys_kin}. This also necessitates to make use of the equation \eqref{pme_sys_kin} in the case of small spatial modes $\xi$ in order to obtain optimal time regularity, see Corollary \ref{cor:av3} below.

  \textbf{Comments on the literature:}  The (spatial) regularity of solutions to porous medium equations in Sobolev spaces has previously been considered in \cite{Ebm05,Ges17,TaT07}. Since our main focus is on time-space regularity, we refer to \cite{Ges17} for a more detailed account on the available literature in this regard.
  
  In the case of non-negative solutions the spatial regularity of special types of powers of solutions has been investigated in the literature. For example, much work is devoted to the pressure defined by $v:=\frac{m}{m-1}u^{m-1}$ (cf.\@ e.g.\@ V\'azquez \cite{V07} and the references therein). In the recent work \cite{GS16}, Gianazza and Schwarzacher proved higher integrability for nonnegative, local weak solutions to forced porous medium equations in the sense that $u^{\frac{m+1}{2}} \in L^{2+\eps}_{loc}((0,T);W^{1,2+\eps}_{loc})$ for all $\eps>0$ small enough. In \cite{BDKS18}, this result was generalized by Bögelein, Duzaar, Korte, and Scheven.
  
  The analysis of regularity in time of solutions to porous medium equations (without forcing) has a long history initiated by Aronson and Bénilan in \cite{AB79} and continued by Crandall, Pazy, Tartar in \cite{CPT79} and Bénilan, Crandall in \cite{BC81}, where it has been shown that 
  \begin{equation}\label{eq:time_l1}
    \partial_t u \in L^1_{loc}((0,\infty);L^1(\R^d))
  \end{equation}
  for $u_0 \in L^1(\R^d)$. Subsequently, in \cite{CP82,CP82-2}, Crandall and Pierre have devoted considerable effort to relax the required assumptions on the nonlinearity $\psi$ in the case of generalized porous medium equations
  \begin{align}\label{gen_pme_sys}
    \partial_t\uf-\Delta \psi(\uf) &= 0 \quad \tin (0,T)\times \R^d.
   \end{align} 
 More precisely, in \cite{CP82} assuming the global generalized homogeneity condition 
 \begin{equation}\label{eq:intro-homogeneity}
   \nu \frac{\psi(r)\psi''(r)}{(\psi'(r))^2}\in [m,M],
 \end{equation}
 for some  $0<m<M$, $\nu \in \set{\pm 1}$ and all $r\in\R$, \eqref{eq:time_l1} was recovered. \\
 It should be noted that the methods developed in these works are restricted to the non-forced case $S\equiv0$. In fact, for $S\not\equiv 0$, the linear case $m=1$ demonstrates that \eqref{eq:time_l1} should not be expected. We are not aware of results proving regularity in time in Sobolev spaces for porous medium equations with non-vanishing forcing. In this sense, restricting to regularity in time alone, the results of the present work can be regarded as the (partial) extension of the results of \cite{AB79,BC81,CPT79,CP82,CP82-2} to non-vanishing forcing.
 
 We are not aware of previous results on mixed time and space regularity in Sobolev spaces for solutions to porous medium equations.
 
 For simplicity of the presentation we restrict to the nonlinearity $\psi(u)=u^{[m]}$ in this work. However, the methods that we present are not restricted to this case, as long as $\psi$ satisfies a non-linearity condition as in \cite{Ges17}. In addition, by means of a velocity decomposition, i.e.\ writing 
    $$u(t,x) = \sum_{i=1}^K u^i(t,x) := \sum_{i=1}^K \int_v \vp^i(v)f(t,x,v) \dd v,$$
 where $\vp^i$, $i=1,\dots,K$ is a smooth decomposition of the unity, such a non-linearity condition only needs to be supposed locally at points of degeneracy. This is in contrast to the assumptions, such as \eqref{eq:intro-homogeneity}, supposed in the series of works \cite{AB79,BC81,CPT79,CP82,CP82-2} mentioned above, which can be regarded as \textit{global} generalized homogeneity conditions.

\textbf{Structure of this work:} In Section \ref{FctSpc} we collect information on the class of homogeneous and inhomogeneous anisotropic, dominating mixed derivative spaces employed in this work. The optimality of the obtained estimates is indicated in Section \ref{Scale} by scaling arguments and by explicit computations in case of the Barenblatt solution. In Section \ref{AvLem} we provide general averaging lemmata (Lemma \ref{lem:av} and Lemma \ref{lem:av2}) in the framework of homogeneous dominating mixed derivative spaces and translate them to more standard inhomogeneous anisotropic fractional Sobolev spaces (Corollary \ref{cor:av0}, Corollary \ref{cor:av} and Corollary \ref{cor:av3}). In this formulation, they imply the main result by their application to the porous medium equation in Section \ref{App_PME}.

\section{Preliminaries, Notation and Function Spaces}\label{FctSpc}

We use the notation $a\lesssim b$ if there is a universal constant $C>0$ such that $a\le Cb$. We introduce $a\gtrsim b$ in a similar manner, and write $a \sim b$ if $a\lesssim b$ and $a\gtrsim b$.
For a Banach space $X$ and $I\subset \R$ we denote by $C(I;X)$ the space of bounded and continuous $X$-valued functions endowed with the norm $\|f\|_{C(I;X)}:=\sup_{t\in I}\|f(t)\|_X$. If $X=\R$ we write $C(I)$. For $k\in \N\cup\set{\infty}$, the space of $k$-times continuously differentiable functions is denoted by $C^k(I;X)$. The subspace of $C^k(I;X)$ consisting of compactly supported functions is denoted by $C^k_c(I;X)$. Moreover, we write $\calm_{TV}$ for the space of all measures with finite total variation.
Throughout this article we use several types of $L^p$-based function spaces.
For a Banach space $X$ and $p\in[1,\infty]$, we endow the Bochner-Lebesgue space $\LR{p}(\R;X)$ with the usual norm
\begin{align*}
 \norm{f}_{\LR{p}(\R;X)}&:=\left(\int_{\R}\norm{f(t)}_{X}^p \dd t\right)^\frac{1}{p},
\end{align*}
with the standard modification in the case of $p=\infty$. For $k\in \N_0:=\N\cup\set{0}$, the corresponding $X$-valued Sobolev space is denoted by $\WSR{k}{p}(\R;X)$. If $\sigma\in (0,\infty)$ is non-integer (say $\sigma=k+r$, with $k\in\N_0$ and $r\in(0,1)$), then we define the $X$-valued Sobolev-Slobodecki\u{\i} space $\WSR{\sigma}{p}(\R;X)$ as the space of functions in $\WSR{k}{p}(\R;X)$ with
\begin{align}\label{hom_sob_norm}
 \norm{f}_{\DSR{\sigma}{p}(\R;X)}:=\left(\int_{\R\times\R}\frac{\norm{D^kf(t)-D^kf(s)}_X^p}{|t-s|^{rp+1}}\dd s \dd t\right)^{\frac{1}{p}}<\infty,
\end{align}
again with the usual modification in the case of $p=\infty$. Further, let $\DSR{\sigma}{p}(\R;X)$ be the space of all locally integrable $X$-valued functions $f$ for which \eqref{hom_sob_norm} is finite. If we factor out the equivalence relation $\sim$, where $f\sim g$ if $\norm{f-g}_{\DSR{\sigma}{p}(\R;X)}=0$, the space $\DSR{\sigma}{p}(\R;X)$ equipped with the norm $\norm{\cdot}_{\DSR{\sigma}{p}(\R;X)}$ is a Banach space.

Moreover, in order to treat regularity results in both time and space efficiently, we introduce spaces with dominating mixed derivatives set in the framework of Fourier analysis, that is, corresponding Besov spaces. These spaces have a long history in the literature, beginning with the works of S.\@ M.\@ Nikol'ski\u{\i} \cite{Nik62, Nik63b, Nik63a}. We refer the reader to the monograph of Schmeisser and Triebel \cite{ScT87} and the references therein. We adopt the notation of \cite{ScT87} for the non-homogeneous spaces. Corresponding homogeneous Besov spaces are treated in \cite{Tri77c, Tri77d}; we adapt the notation to be consistent with the one in \cite{ScT87}.
We recall from \cite{Tri77c} the definition of the spaces $\calz$ and $\calz'$ replacing the standard Schwartz space $\cals=\cals(\R^{d+1})$ and the space of tempered distributions $\cals'=\cals'(\R^{d+1})$ in the definition of homogeneous spaces. As we are concerned with function spaces in the time variable $t\in \R$ and the spatial variable $x\in \R^d$, we introduce besides $\R^{d+1}=\R_t\times \R^d_x$ also the subset
\begin{align*}
 \dot\R^{d+1}:=\setc{(t,x)\in\R^{d+1}}{t|x|\neq 0}.
\end{align*}
Note that in \cite{Tri77c}, the notation $\overset{+}{\R^2}$ is used, which gives a better geometrical intuition of the set taken out of $\R^{2}$. However, for typesetting reasons, we have decided for the notation $\dot\R^{d+1}$. Then we let $\dot\cald$ be the subset of the standard space of test functions $\cald$, consisting of functions with compact support in $\dot\R^{d+1}$ and view it as a locally convex space equipped with the canonical topology. Its dual space is denoted by $\dot\cald'$, and is referred to as distributions over $\dot\R^{d+1}$. We define $\calz$ as the image of $\dot\cald\subset \cals$ under the Fourier transform $\calf$ in time and space, equipped with the inherited topology from $\dot\cald$. The corresponding dual space is denoted by $\calz'$. Since $\calf:\dot\cald\to\calz$, we can define by duality the Fourier transform $\calf:\calz'\to\dot\cald'$.

It holds $\calz\subset \cals$ with a continuous embedding, but the fact that $\calz$ is not densely embedded in $\cals$ prevents one from stating $\cals'\subset \calz'$. However, we note that for $p\in(1,\infty)$, the space $\LR{p}(\R^{d+1})$ can be viewed both as subspace of $\cals'$ and as a subspace of $\calz'$, cf.\@ Theorem 3.3 in \cite{Tri77c}.

Let $\vp$ be a smooth function supported in the annulus $\{\xi\in\R^{d}:\,\frac{1}{2}\le|\xi|\le2\}$ and such that
\[
\sum_{j\in\Z}\vp_j(\xi):=\sum_{j\in\Z}\vp(2^{-j}\xi)=1,\quad\forall\xi\in\R^{d}\setminus\set{0}.
\]
Similarly, let $\eta$ be a smooth function supported in $\vpp{-2,-\frac12}\cup \vpp{\frac12,2}$ with
\[
\sum_{l\in\Z}\eta_l(\tau):=\sum_{l\in\Z}\eta(2^{-l}\tau)=1,\quad\forall\tau\in\R\setminus\set{0}.
\]
Moreover, define $\phi_j:=\vp_j$ for $j\ge 1$ and $\phi_0:=1-\sum_{j\ge 1}\phi_j$ as well as $\psi_l:=\eta_l$ for $l\ge 1$ and $\psi_0:=1-\sum_{l\ge 1}\eta_l$. We will use the shorthand notation $\eta_l\vp_j$ for the function $(\tau,\xi)\mapsto \eta_l(\tau)\vp_j(\xi)$, and similarly for combinations of $\psi_l$ and $\phi_j$.

\begin{defn}\label{defn:spaces}
 Let $\sigma_i\in\vpp{-\infty,\infty}$, $i=t,x$, and $p\in[1,\infty]$. Set $\overline{\sigma}:=(\sigma_t,\sigma_x)$.
 \begin{enumerate}
  \item The homogeneous Besov space with dominating mixed derivatives $S^{\overline{\sigma}}_{p,\infty}\dot B(\R^{d+1})$ is given by
  \begin{align*}
   S^{\overline{\sigma}}_{p,\infty}\dot B:=S^{\overline{\sigma}}_{p,\infty}\dot B(\R^{d+1}):=\setc{f\in\calz'}{\|f\|_{S^{\overline{\sigma}}_{p,\infty}\dot B}<\infty},
  \end{align*}
  with the norm
  \begin{align*}
   \|f\|_{S^{\overline{\sigma}}_{p,\infty}\dot B}:=\sup_{l,j\in\Z} 2^{\sigma_t l}2^{\sigma_x j}\|\calf^{-1}_{t,x}\eta_l \vp_j\calf_{t,x} f\|_{\LR{p}(\R^{d+1})}.
  \end{align*}
  Similarly, the space $S^{\overline{\sigma}}_{p,\infty,(\infty)}\dot B(\R^{d+1})$ is given via the norm
  \begin{align*}
   \|f\|_{S^{\overline{\sigma}}_{p,\infty,(\infty)}\dot B}:=\sup_{l,j\in\Z} 2^{\sigma_t l}2^{\sigma_x j}\|\calf^{-1}_{t,x}\eta_l \vp_j\calf_{t,x} f\|_{\LR{p,\infty}(\R^{d+1})}.
  \end{align*}
  \item The homogeneous Chemin-Lerner spaces $\td{L}^{p}_{t}\dot B^{\sigma_x}_{p,\infty}(\R^{d+1})$ and $\td{L}^{p}_{x}\dot B^{\sigma_t}_{p,\infty}(\R^{d+1})$ are given by
  \begin{align*}
   \td{L}^{p}_{t}\dot B^{\sigma_x}_{p,\infty}&:=\td{L}^{p}_{t}\dot B^{\sigma_x}_{p,\infty}(\R^{d+1}):=\setc{f\in\cals'}{\|f\|_{\td{L}^{p}_{t}\dot B^{\sigma_x}_{p,\infty}}<\infty}, \\
   \td{L}^{p}_{x}\dot B^{\sigma_t}_{p,\infty}&:=\td{L}^{p}_{x}\dot B^{\sigma_t}_{p,\infty}(\R^{d+1}):=\setc{f\in\cals'}{\|f\|_{\td{L}^{p}_{x}\dot B^{\sigma_t}_{p,\infty}}<\infty},
  \end{align*}
  with the norms
  \begin{align*}
   \|f\|_{\td{L}^{p}_{t}\dot B^{\sigma_x}_{p,\infty}}&:=\sup_{j\in\Z} 2^{\sigma_x j}\|\calf^{-1}_{x} \vp_j\calf_{x} f\|_{\LR{p}(\R^{d+1})}, \\
   \|f\|_{\td{L}^{p}_{x}\dot B^{\sigma_t}_{p,\infty}}&:=\sup_{l\in\Z} 2^{\sigma_t l}\|\calf^{-1}_{t} \eta_l\calf_{t} f\|_{\LR{p}(\R^{d+1})},
  \end{align*}
  respectively.
  \item The non-homogeneous Besov space with dominating mixed derivatives $S^{\overline{\sigma}}_{p,\infty}B(\R^{d+1})$ is given by
  \begin{align*}
   S^{\overline{\sigma}}_{p,\infty}B:=S^{\overline{\sigma}}_{p,\infty}B(\R^{d+1}):=\setc{f\in\cals'(\R^{d+1})}{\|f\|_{S^{\overline{\sigma}}_{p,\infty} B}<\infty},
  \end{align*}
  with the norm
  \begin{align*}
   \|f\|_{S^{\overline{\sigma}}_{p,\infty}B}:=\sup_{l,j\ge 0} 2^{\sigma_t l}2^{\sigma_x j}\|\calf^{-1}_{t,x}\psi_l \phi_j\calf_{t,x} f\|_{\LR{p}(\R^{d+1})}.
  \end{align*}
  \item The non-homogeneous Chemin-Lerner space $\td{L}^{p}_{t} B^{\sigma_x}_{p,\infty}(\R^{d+1})$ is given by
  \begin{align*}
   \td{L}^{p}_{t} B^{\sigma_x}_{p,\infty}&:=\td{L}^{p}_{t} B^{\sigma_x}_{p,\infty}(\R^{d+1}):=\setc{f\in\cals'}{\|f\|_{\td{L}^{p}_{t} B^{\sigma_x}_{p,\infty}}<\infty},
  \end{align*}
  with the norm
  $\displaystyle \|f\|_{\td{L}^{p}_{t} B^{\sigma_x}_{p,\infty}}:=\sup_{j\ge 0} 2^{\sigma_x j}\|\calf^{-1}_{x} \phi_j\calf_{x} f\|_{\LR{p}(\R^{d+1})}$.
 \end{enumerate}
\end{defn}

\begin{rem}
 All spaces considered in Definition \ref{defn:spaces} are Banach spaces, cf.\@ \cite{Tri77c}. Note that for $\vt\in\R$, we use the notation $\vt\overline{\sigma}=(\vt\sigma_t,\vt\sigma_x)$. In this note, we restrict ourselves to the third index of the Besov-type space being infinity, in which case the spaces $S^{\overline{\sigma}}_{p,\infty}B$ are sometimes called \emph{Nikol'ski\u{\i} spaces of dominating mixed derivatives} in the literature. However, there is no conceptual limitation to consider also third indices $q\in[1,\infty]$. On the same token, one could also consider spaces with different indices $p$ and $q$ in different directions. We refer the reader to the monograph \cite{ScT87} for more details concerning such spaces.
\end{rem}

\begin{lem}\label{lem:emb_0}
 Let $\kappa_x\ge 0$ and $p\in [1,\infty]$. Then 
 \[
 \td{L}^{p}_{t} B^{\kappa_x+\eps}_{p,\infty}(\R^{d+1})\subset \LR{p}(\R;\WSR{\kappa_x}{p}(\R^d))\subset \td{L}^{p}_{t} B^{\kappa_x-\delta}_{p,\infty}(\R^{d+1}),
 \]
  whenever $\eps>0$ and $\delta\in(0,\kappa_x]$.
\end{lem}
\begin{proof}
 Follows from \cite[p.\@ 98]{BCD11}.
\end{proof}

\begin{lem}\label{lem:emb_1}
 Let $\kappa_t,\kappa_x>0$ and $p\in [1,\infty)$. Then $S^{\overline{\kappa}}_{p,\infty}B\subset \WSR{\sigma_t}{p}(\R;\WSR{\sigma_x}{p}(\R^d))$ whenever $\sigma_t\in[0,\kappa_t)$ and $\sigma_x\in[0,\kappa_x)$.
\end{lem}
\begin{proof}
 The proof is a combination of results in \cite{ScT87}, which are written for $\R\times \R$ but also true for $\R\times \R^d$ by an inspection of their respective proofs: Without loss of generality, we can assume that $\sigma_t$ and $\sigma_x$ are non-integer. By \cite[Remark 2.3.4/4]{ScT87}, we have $\WSR{\sigma_t}{p}(\R;\WSR{\sigma_x}{p}(\R^d)) = S B^{\overline{\sigma}}_{p,p}$ (see \cite[Definition 2.2.1/2]{ScT87} for a definition of the latter space). Since by \cite[Proposition 2.2.3/2]{ScT87} we have $S^{\overline{\kappa}}_{p,\infty} B\subset S B^{\overline{\sigma}}_{p,p}$, this yields the claim.
\end{proof}

\begin{lem}\label{lem:emb_2}
 Let $\sigma_t,\sigma_x> 0$ and $p\in[1,\infty]$. Then
  \begin{align*}
   \left(\LR{p}(\R^{d+1})\cap \td{L}^{p}_{x}\dot B^{\sigma_t}_{p,\infty} \cap \td{L}^{p}_{t} \dot B^{\sigma_x}_{p,\infty} \cap S^{\overline{\sigma}}_{p,\infty}\dot B\right) = S^{\overline{\sigma}}_{p,\infty} B
  \end{align*}
  with equivalent norms.
\end{lem}
\begin{proof}
 As smooth and compactly supported functions, $\psi_0$ and $\phi_0$ extend to $L^p$ multipliers for all $p\in[1,\infty]$, see e.g.\@ \cite{BeL76}.\newline
 For $f\in\left(\LR{p}(\R^{d+1})\cap \td{L}^{p}_{x}\dot B^{\sigma_t}_{p,\infty} \cap \td{L}^{p}_{t} \dot B^{\sigma_x}_{p,\infty} \cap S^{\overline{\sigma}}_{p,\infty}\dot B\right)\subset\cals'(\R^{d+1})$ we obtain
  \begin{align*}
   \|f\|_{S^{\overline{\sigma}}_{p,\infty}B}&\le \|\calf^{-1}_{t,x}\psi_0 \phi_0\calf_{t,x} f\|_{\LR{p}_{t,x}} + \sup_{l> 0} 2^{\sigma_t l}\|\calf^{-1}_{t,x}\eta_l \phi_0\calf_{t,x} f\|_{\LR{p}_{t,x}} \\
   &\quad + \sup_{j> 0} 2^{\sigma_x j}\|\calf^{-1}_{t,x}\psi_0 \vp_j\calf_{t,x} f\|_{\LR{p}_{t,x}} + \sup_{l,j> 0} 2^{\sigma_t l}2^{\sigma_x j}\|\calf^{-1}_{t,x}\eta_l \vp_j\calf_{t,x} f\|_{\LR{p}_{t,x}} \\
   &\lesssim \|f\|_{\LR{p}_{t,x}} + \sup_{l> 0} 2^{\sigma_t l}\|\calf^{-1}_{t}\eta_l \calf_{t} f\|_{\LR{p}_{t,x}} \\
   &\quad + \sup_{j> 0} 2^{\sigma_x j}\|\calf^{-1}_{x} \vp_j\calf_{x} f\|_{\LR{p}_{t,x}} + \sup_{l,j> 0} 2^{\sigma_t l}2^{\sigma_x j}\|\calf^{-1}_{t,x}\eta_l \vp_j\calf_{t,x} f\|_{\LR{p}_{t,x}} \\
   &\lesssim \|f\|_{\LR{p}_{t,x}} + \|f\|_{\td{L}^{p}_{x}\dot B^{\sigma_t}_{p,\infty}} + \|f\|_{\td{L}^{p}_{t}\dot B^{\sigma_x}_{p,\infty}} + \|f\|_{S^{\overline{\sigma}}_{p,\infty}\dot B}.
  \end{align*}
 Conversely, for $f\in S^{\overline{\sigma}}_{p,\infty}B$, we estimate the four contributions corresponding to $\LR{p}(\R^{d+1})$, $\td{L}^{p}_{x}\dot B^{\sigma_t}_{p,\infty}$, $\td{L}^{p}_{t} \dot B^{\sigma_x}_{p,\infty}$, and $S^{\overline{\sigma}}_{p,\infty}\dot B$ separately. We start by noting that due to $\sigma_t,\sigma_x>0$, the invariance of multiplier norms with respect to dilation, $\eta_l=\eta_l\td\psi_0$ for $l\le 0$ and $\vp_j=\vp_j\td\phi_0$ for $j\le 0$, where $\td\psi_0:=\psi_0+\psi_1$ and $\td\phi_0:=\phi_0+\phi_1$, we have
 \begin{align*}
  \sup_{l\le 0}2^{\sigma_t l}\|\calf^{-1}_{t} \eta_l\calf_{t} f\|_{\LR{p}_{t,x}}&\lesssim \|\calf^{-1}_{t} \td\psi_0\calf_{t} f\|_{\LR{p}_{t,x}}, \\
  \sup_{j\le 0}2^{\sigma_x j}\|\calf^{-1}_{x} \vp_j\calf_{x} f\|_{\LR{p}_{t,x}}&\lesssim \|\calf^{-1}_{x} \td\phi_0\calf_{x} f\|_{\LR{p}_{t,x}}.
 \end{align*}
 Furthermore we use the fact that for $\sigma>0$ one has the estimate $\sum_{n\ge 0}|a_n|\lesssim \sup_{n\ge 0} 2^{\sigma n}|a_n|$ for any sequence $(a_n)\subset \R$ with a constant depending on $\sigma$.
 With this, we obtain
 \begin{align*}
  \|f\|_{\LR{p}_{t,x}}&\le \sum_{l,j\ge 0} \|\calf^{-1}_{t,x}\psi_l \phi_j\calf_{t,x} f\|_{\LR{p}_{t,x}} \lesssim \sup_{l,j\ge 0} 2^{\sigma_t l}2^{\sigma_x j}\|\calf^{-1}_{t,x}\psi_l \phi_j\calf_{t,x} f\|_{\LR{p}_{t,x}}\le \|f\|_{S^{\overline{\sigma}}_{p,\infty}B}.
 \end{align*}
 Next, we compute
 \begin{align*}
  \|f\|_{\td{L}^{p}_{x}\dot B^{\sigma_t}_{p,\infty}}&\le \sup_{l\le 0} 2^{\sigma_t l}\|\calf^{-1}_{t}\eta_l \calf_{t} f\|_{\LR{p}_{t,x}} + \sup_{l> 0} 2^{\sigma_t l}\|\calf^{-1}_{t}\psi_l \calf_{t} f\|_{\LR{p}_{t,x}} \\
  &\lesssim \|\calf^{-1}_{t}\td\psi_0 \calf_{t} f\|_{\LR{p}_{t,x}} + \sup_{l> 0} 2^{\sigma_t l}\|\calf^{-1}_{t}\psi_l \calf_{t} f\|_{\LR{p}_{t,x}} \\
  &\le \sum_{j\ge 0} \|\calf^{-1}_{t,x}\td\psi_0 \phi_j\calf_{t,x} f\|_{\LR{p}_{t,x}} + \sup_{l>0}\sum_{j\ge 0} 2^{\sigma_t l}\|\calf^{-1}_{t,x}\psi_l \phi_j\calf_{t,x} f\|_{\LR{p}_{t,x}} \\
  &\lesssim \sup_{j\ge 0} 2^{\sigma_x j} \|\calf^{-1}_{t,x}\td\psi_0 \phi_j\calf_{t,x} f\|_{\LR{p}_{t,x}} + \sup_{l>0,j\ge 0} 2^{\sigma_t l}2^{\sigma_x j}\|\calf^{-1}_{t,x}\psi_l \phi_j\calf_{t,x} f\|_{\LR{p}_{t,x}} \lesssim \|f\|_{S^{\overline{\sigma}}_{p,\infty}B}.
 \end{align*}
 By analogy, $\|f\|_{\td{L}^{p}_{t}\dot B^{\sigma_x}_{p,\infty}}\lesssim \|f\|_{S^{\overline{\sigma}}_{p,\infty}B}$. Hence, it remains to control $\|f\|_{S^{\overline{\sigma}}_{p,\infty}\dot B}$. We split this term into the four contributions
 \begin{align*}
  \|f\|_{S^{\overline{\sigma}}_{p,\infty}\dot B}&=\sup_{l,j> 0} 2^{\sigma_t l}2^{\sigma_x j}\|\calf^{-1}_{t,x}\psi_l \phi_j\calf_{t,x} f\|_{\LR{p}_{t,x}} + \sup_{l>0,j\le 0} 2^{\sigma_t l}2^{\sigma_x j}\|\calf^{-1}_{t,x}\psi_l \vp_j\calf_{t,x} f\|_{\LR{p}_{t,x}} \\
  &\quad + \sup_{l\le 0, j>0} 2^{\sigma_t l}2^{\sigma_x j}\|\calf^{-1}_{t,x}\eta_l \phi_j\calf_{t,x} f\|_{\LR{p}_{t,x}} + \sup_{l,j\le 0} 2^{\sigma_t l}2^{\sigma_x j}\|\calf^{-1}_{t,x}\eta_l \vp_j\calf_{t,x} f\|_{\LR{p}_{t,x}}.
 \end{align*}
 The first contribution is immediately estimated by $\|f\|_{S^{\overline{\sigma}}_{p,\infty}B}$. For the second contribution, we have
 \begin{align*}
  \sup_{l>0,j\le 0} 2^{\sigma_t l}2^{\sigma_x j}\|\calf^{-1}_{t,x}\psi_l \vp_j\calf_{t,x} f\|_{\LR{p}_{t,x}}\lesssim \sup_{l>0} 2^{\sigma_t l}\|\calf^{-1}_{t,x}\psi_l \td\phi_0\calf_{t,x} f\|_{\LR{p}_{t,x}}\le \|f\|_{S^{\overline{\sigma}}_{p,\infty}B},
 \end{align*}
 and a similar estimate holds for the third contribution. For the fourth contribution, we have
 \begin{align*}
  \sup_{l,j\le 0} 2^{\sigma_t l}2^{\sigma_x j}\|\calf^{-1}_{t,x}\eta_l \vp_j\calf_{t,x} f\|_{\LR{p}_{t,x}}\lesssim \|\calf^{-1}_{t,x}\td\psi_0 \td\phi_0\calf_{t,x} f\|_{\LR{p}_{t,x}}.
 \end{align*}
 This concludes the proof.
\end{proof}

\section{Optimality of Estimates via Scaling}\label{Scale}
It is well known that in the linear case $m=1$ one has estimates of the form
\begin{align}\label{scale_lin_lp_est}
 \norm{\uf}_{\LR{1}_t \DSR{\sigma_x}{1}_x}\le c(\sigma_x) \vpp{\norm{\uf_0}_{\LR{1}_x} + \norm{S}_{\LR{1}_{t,x}}},
\end{align}
for all $\sigma_x< 2$. In the case $m>1$, such an estimate cannot be true for \emph{any} $\sigma_x>0$ anymore. Intuitively, this is due to the linear nature of \eqref{scale_lin_lp_est} (observe that the integrability exponent is equal on both sides of the inequality), which is not compatible with the nonlinear equation \eqref{pme_sys}. We will make this intuition more precise by the following lemma based on a scaling argument.

\begin{lem}\label{lem:scaling}
 Let $T>0$, $m>1$, $\mu\in [1,m]$, $p\in [1,\infty)$ and $\sigma_t,\sigma_x \ge 0$. Assume that there is a constant $c=c(m,\mu,p,\sigma_t,\sigma_x)>0$ such that
\begin{align}\label{scale_pme_lp_est}
 \norm{\uf^{[\mu]}}_{\DSR{\sigma_t}{p}(0,T;\DSR{\sigma_x}{p}(\R^d))}^p\le c \vpp{\norm{\uf_0}_{\LR{1}(\R^d)} + \norm{S}_{\LR{1}(0,T;\LR{1}(\R^d))}}
\end{align}
for all solutions $\uf$ to \eqref{pme_sys}. Then 
\begin{align}\label{scaling_const}
\begin{split}
 p &\le \frac{m}{\mu+(m-1)\sigma_t}\le \frac{m}{\mu}, \\
 \sigma_t&\le \frac{m-\mu p}{p(m-1)}\le \frac{m-\mu}{m-1}, \quad \tand\\
 \sigma_x &= \frac{\mu p -1}{p}\frac{2}{m-1}\le \frac{2(\mu-\sigma_t)}{m}\le \frac{2\mu}{m}. 
\end{split}
\end{align}
In particular, if $\sigma_t=\frac{m-\mu}{m-1}$, then $p=1$ and $\sigma_x= \frac{2(\mu-1)}{m-1}$.
\end{lem}
\begin{proof}
 For positive constants $\eta,\gamma \ge 1$ with $\eta^{m-1}=\gamma$ and a fixed triple $(\uf,\uf_0,S)$ such that $\uf$ satisfies \eqref{pme_sys} with initial condition $\uf_0$ and forcing $S$ we consider the rescaled quantities $(\ti\uf, \ti\uf_0, \ti S)$ defined via
\begin{align*}
 \ti\uf(t,x):= \eta \uf(\gamma t, x), \quad \ti \uf_0(x):= \eta \uf_0 (x), \quad \ti S(t,x):= \eta^m S(\gamma t, x),
\end{align*}
where we have tacitly extended $S$ on $(T,\gamma T)$ by $0$.
Then $\ti\uf$ satisfies \eqref{pme_sys} with $\ti \uf_0\in \LR{1}(\R^d)$ and $\ti S\in \LR{1}(0,T;\LR{1}(\R^d))$, so that \eqref{scale_pme_lp_est} gives
\begin{align}\label{rescaled_pme_lp_est_time}
 \norm{\ti\uf^{[\mu]}}_{\DSR{\sigma_t}{p}(0,T;\DSR{\sigma_x}{p}(\R^d))}^p\le c \vpp{\norm{\ti \uf_0}_{\LR{1}(\R^d)} + \norm{\ti S}_{\LR{1}(0,T;\LR{1}(\R^d))}}.
\end{align}
We observe
\begin{align*}
 \norm{\ti\uf^{[\mu]}}_{\DSR{\sigma_t}{p}(0,T;\DSR{\sigma_x}{p}(\R^d))}^p =\eta^{\mu p} \gamma^{\sigma_t p -1} \norm{\uf^{[\mu]}}_{\DSR{\sigma_t}{p}(0,\gamma T;\DSR{\sigma_x}{p}(\R^d))}^p
\end{align*}
as well as $\norm{\ti \uf_0}_{\LR{1}(\R^d)} = \eta \norm{\uf_0}_{\LR{1}(\R^d)}$ and $\norm{\ti S}_{\LR{1}(0,T;\LR{1}(\R^d))}= \eta  \norm{S}_{\LR{1}(0,\gamma T;\LR{1}(\R^d))}$. Thus, it follows from \eqref{rescaled_pme_lp_est_time} that
\begin{align}\label{scale_pme_lp_est_eta_time}
\begin{split}
 &\norm{\uf^{[\mu]}}_{\DSR{\sigma_t}{p}(0,T;\DSR{\sigma_x}{p}(\R^d))}^p \\
 &\qquad \le c \eta^{1- \mu p} \gamma^{1-\sigma_t p} \vpp{\norm{\uf_0}_{\LR{1}(\R^d)} + \norm{S}_{\LR{1}(0,\gamma T;\LR{1}(\R^d))}} \\
 &\qquad = c \eta^{(m-1)(1-\sigma_t p) + 1-\mu p} \vpp{\norm{\uf_0}_{\LR{1}(\R^d)} + \norm{S}_{\LR{1}(0,T;\LR{1}(\R^d))}}.
\end{split}
\end{align}
As long as $\uf_0$ or $S$ are non-trivial and unless
\begin{align}\label{scale_exp_cond_time}
 (m-1)(1-\sigma_t p) + 1-\mu p \geq 0,
\end{align}
this gives the contradiction $\uf=0$ by sending $\eta\to\infty$ (and consequently also $\gamma\to \infty$). Since $\sigma_t\ge 0$, \eqref{scale_exp_cond_time} gives $p\le \frac{m}{\mu+(m-1)\sigma_t}\le \frac{m}{\mu}$. On the same token, since $p\ge 1$, \eqref{scale_exp_cond_time} gives $\sigma_t \le \frac{m-\mu p}{p(m-1)} \le \frac{m-\mu}{m-1}$.

Next, we rescale in space. More precisely, for positive constants $\eta,\gamma>0$ with $\eta^{1-m}=\gamma^2$ and a fixed triple $(\uf,\uf_0,S)$ as above we consider the rescaled quantities $(\ti\uf, \ti\uf_0, \ti S)$ defined via
\begin{align*}
 \ti\uf(t,x):= \eta \uf(t, \gamma x), \quad \ti \uf_0(x):= \eta \uf_0 (\gamma x), \quad \ti S(t,x):= \eta S(t,\gamma x).
\end{align*}
Then $\ti\uf$ satisfies \eqref{pme_sys} with $\ti \uf_0\in \LR{1}(\R^d)$ and $\ti S\in \LR{1}(0,T;\LR{1}(\R^d))$, so that \eqref{scale_pme_lp_est} gives
\begin{align}\label{rescaled_pme_lp_est}
 \norm{\ti\uf^{[\mu]}}_{\DSR{\sigma_t}{p}(0,T;\DSR{\sigma_x}{p}(\R^d))}^p\le c \vpp{\norm{\ti \uf_0}_{\LR{1}(\R^d)} + \norm{\ti S}_{\LR{1}(0,T;\LR{1}(\R^d))}}.
\end{align}
We have
\begin{align*}
 \norm{\ti\uf^{[\mu]}}_{\DSR{\sigma_t}{p}(0,T;\DSR{\sigma_x}{p}(\R^d))}^p =\eta^{\mu p} \gamma^{\sigma_x p -d} \norm{\uf^{[\mu]}}_{\DSR{\sigma_t}{p}(0,T;\DSR{\sigma_x}{p}(\R^d))}^p
\end{align*}
as well as $\norm{\ti \uf_0}_{\LR{1}(\R^d)} = \eta \gamma^{-d} \norm{\uf_0}_{\LR{1}(\R^d)}$ and $\norm{\ti S}_{\LR{1}(0,T;\LR{1}(\R^d))}= \eta \gamma^{-d} \norm{S}_{\LR{1}(0,T;\LR{1}(\R^d))}$. Thus, it follows from \eqref{rescaled_pme_lp_est} and the relation $\eta^{1-m}=\gamma^2$ that
\begin{align}\label{scale_pme_lp_est_eta}
\begin{split}
 &\norm{\uf^{[\mu]}}_{\DSR{\sigma_t}{p}(0,T;\DSR{\sigma_x}{p}(\R^d))}^p \\
 &\qquad \le c \eta^{1- \mu p} \gamma^{-\sigma_x p} \vpp{\norm{\uf_0}_{\LR{1}(\R^d)} + \norm{S}_{\LR{1}(0,T;\LR{1}(\R^d))}} \\
 &\qquad = c \eta^{\frac{\sigma_x p(m-1)}{2} + 1-\mu p} \vpp{\norm{\uf_0}_{\LR{1}(\R^d)} + \norm{S}_{\LR{1}(0,T;\LR{1}(\R^d))}}.
\end{split}
\end{align}
As long as $\uf_0$ or $S$ are non-trivial and unless

\begin{align}\label{scale_exp_cond}
 \frac{\sigma_x p (m-1)}{2} + 1- \mu p = 0 \Leftrightarrow \sigma_x = \frac{\mu p-1}{p} \frac{2}{m-1},
\end{align}
this gives the contradiction $\uf=0$ by sending $\eta\to 0$ or $\eta\to\infty$ (and consequently $\gamma\to \infty$ or $\gamma\to 0$, respectively). Plugging into \eqref{scale_exp_cond} the restrictions on $p$ and $\sigma_t$, we obtain the result.
\end{proof}

\begin{rem}
 If one sets $\mu=1$, $p=1$ and $\sigma_t=0$, Lemma \ref{lem:scaling} tells us that $\sigma_x$ cannot be positive, which is what we claimed following \eqref{scale_lin_lp_est}. Moreover, we emphasize that Lemma \ref{lem:scaling} shows that in the case of the whole space, the regularity exponent $\sigma_x\in [\frac{2(\mu-1)}{m-1},\frac{2\mu}{m}]$ is in a one-to-one correspondence to the integrability exponent $p\in[1,\frac{m}{\mu}]$ via
 \begin{align*}
  \sigma_x= \frac{\mu p -1}{p}\frac{2}{m-1}, \quad \tand \quad p= \frac{2}{2\mu-\sigma_x(m-1)}.
 \end{align*}
\end{rem}

\subsection{The Barenblatt Solution}\label{Barenblatt}

Consider the Barenblatt solution
\begin{align*}
 \ubb(t,x):=t^{-\alpha}(C-k|xt^{-\beta}|^2)_+^{\frac{1}{m-1}},
\end{align*}
where $m>1$, $\alpha:=\frac{d}{d(m-1)+2}$, $k=\frac{\alpha(m-1)}{2md}$, $\beta=\frac{\alpha}{d}$, and $C>0$ is a free constant. Then, for $\mu\in [1,m]$,
%
 $\displaystyle \ubb^{[\mu]}\in L^{\frac{m}{\mu}}(0,T;\DSR{s}{\frac{m}{\mu}}(\R^d))$ 
%
implies $s<\frac{2\mu}{m}$.

\begin{proof}
 With $F(x):=(C-k|x|^2)_+^{\frac{\mu}{m-1}}$ we have $\ubb^{[\mu]}(t,x)=t^{-\alpha\mu}F(xt^{-\beta})$. We next observe that, for $s\in (0,1)$ and each $t\ge 0$,
 \begin{align*}
  \norm{\ubb^{[\mu]}(t,\cdot)}_{\DSR{s}{\frac{m}{\mu}}(\R^d)}^{\frac{m}{\mu}} & =\int_{\R^d\times\R^d} \frac{|\ubb^{[\mu]}(t,x)-\ubb^{[\mu]}(t,y)|^{\frac{m}{\mu}}}{|x-y|^{\frac{sm}{\mu}+d}}\dd x \dd y \\
  &=t^{-\alpha m - \beta(\frac{sm}{\mu}+d) + 2d\beta} \norm{F}_{\DSR{s}{\frac{m}{\mu}}(\R^d)}^{\frac{m}{\mu}}.
 \end{align*}
 Hence,
 \begin{align*}
  \norm{\ubb^{[\mu]}}_{L^{\frac{m}{\mu}}(0,T;\DSR{s}{\frac{\mu}{m}}(\R^d))}^{\frac{m}{\mu}} =\norm{t^{-\alpha m - \beta(\frac{sm}{\mu}+d) + 2d\beta}}_{L^1(0,T)} \norm{F}_{\DSR{s}{\frac{m}{\mu}}(\R^d)}^{\frac{m}{\mu}},
 \end{align*}
 which is finite if and only if
 \begin{align*}
  -\alpha m - \beta(\frac{sm}{\mu}+d) + 2d\beta > -1 \quad \tand \quad F\in \DSR{s}{\frac{m}{\mu}}(\R^d).
 \end{align*}
 Hence, necessarily
 \begin{align*}
  m + \frac1d(\frac{sm}{\mu}+d) - 2 < \frac{1}{\alpha} = \frac{d(m-1)+2}{d},
 \end{align*}
 which is equivalent to $s<\frac{2\mu}{m}$. In the case $s\in (1,2)$ we observe that it holds $\partial_{x_i}\ubb^{[\mu]}(t,x)=t^{-\alpha\mu+\beta}\partial_{x_i}F(xt^{-\beta})$, so that analogous arguments may be applied.
\end{proof}

\section{Averaging Lemma Approach}\label{AvLem}

In \cite{Ges17}, an Averaging Lemma has been introduced that can be applied directly to the porous medium equations \eqref{pme_sys} to obtain estimates on the spatial regularity of $\uf$, but so far, no corresponding estimates for powers of the solution $\uf^{\mu}$ or its time regularity could be obtained. In this section, we provide an Averaging Lemma that gives a comprehensive answer to both of these questions. To this end, we recall the definition of the anisotropic and isotropic truncation properties from \cite{Ges17}, which extend the truncation property introduced in \cite[Definition 2.1]{TaT07}.

\begin{defn}\label{def:tuncation_prop}\mbox{ }
\begin{enumerate}
\item Let $m$ be a complex-valued Fourier multiplier. We say that $m$ has the truncation property if, for any locally supported bump function $\psi$ on $\C$ and any $1\le p<\infty$, the multiplier with symbol $\psi(\frac{m(\xi)}{\d})$ is an $L^{p}$-multiplier as well as an $\calm_{TV}$-multiplier uniformly in $\delta>0$, that is, its $L^{p}$-multiplier norm ($\calm_{TV}$-multiplier norm resp.) depends only on the support and $C^{l}$ size of $\psi$ (for some large $l$ that may depend on $m$) but otherwise is independent of $\delta$.
\item Let $m:\R_{\xi}^{d}\times\R_{v}\to\C$ be a Carath\'{e}odory function such that $m(\cdot,v)$ is radial for all $v\in\R$. Then $m$ is said to satisfy the isotropic truncation property if for every bump function $\psi$ supported on a ball in $\C$, every bump function $\vp$ supported in $\{\xi\in\C:\,\frac12\le|\xi|\le2\}$ and every $1<p<\infty$ 
\[
M_{\psi,J}f(x,v):=\mcF_{x}^{-1}\vp\left(\frac{|\xi|^{2}}{J^{2}}\right)\psi\left(\frac{m(\xi,v)}{\d}\right)\mcF_{x}f(x)
\]
is an $L_{x}^{p}$-multiplier for all $v\in\R$, $J=2^{j},\,j\in\Z$ and, for all $r\ge1$,
\begin{align*}
\Big\|\|M_{\psi,J}\|_{\calm^{p}}\Big\|_{L_{v}^{r}} & \lesssim|\Omega_{m}(J,\d)|^{\frac{1}{r}},
\end{align*}
where
$\displaystyle 
\Omega_{m}(J,\d):=\{v\in\R:\,|\frac{m(J,v)}{\d}|\in\supp\psi\}$.
Here we use an abuse of notation  $|\frac{m(J,v)}{\d}|:=\sup\setc{|\frac{m(\xi,v)}{\d}|}{\frac{|\xi|^2}{J^2}\in\supp \vp}$.
\end{enumerate}
\end{defn}

We recall that for $m(\xi,v):=|\xi|^2 |v|$, the anisotropic truncation property is satisfied uniformly in $v$ by Example A.2 and the isotropic truncation property is satisfied by Example 3.2 in \cite{Ges17}, albeit only in the case $J\ge 1$. However, the proof given there can be used without any changes to obtain the full assertion for general $J\in\Z$.

\begin{lem}\label{lem:av}
 Assume $m\in \vpp{1,\infty}$, $\gamma\in\vpp{-\infty,m}$, $\mu\in [1,m+1-\gamma)$ and let $f\in \LR{\rhon}_{t,x,v}$, where $\rhon'=\frac{1}{\rho}$ with $\rho\in(0,1)$, be a solution to
 \begin{align}\label{av_eqn}
  \diffop(\partial_t,\nabla_x,\vf)f(t,x,\vf)=\g_0(t,x,\vf)+\partial_\vf \g_1(t,x,\vf)  \  \ton \ \R_t \times \R^d_x \times \R_\vf.
 \end{align}
 Here, the differential operator $\diffop(\partial_t,\nabla_x,\vf)$ that is given in terms of its symbol
 \begin{align}\label{av_op}
  \Fdiffop &:= i\tau
  + |v|^{m-1}|\xi|^2,
 \end{align}
 and $g_i$ are Radon measures satisfying
 \begin{align*}
|g_{0}|(t,x,v)|v|^{1-\gamma}+|g_{1}|(t,x,v)|v|^{-\gamma}\in
\calm_{TV}(\R_{t}\times\R_{x}^{d}\times\R_{v}).
 \end{align*}
 Suppose $s\in (\frac{\mu-2+\gamma}{m-1},1]\cap [0,1]$. Then $\overline{f}\in S^{\overline{\kappa}}_{p,\infty,(\infty)}\dot B$, where\newline $\overline{f}(t,x):=\int f(t,x,\vf) |\vf|^{\mu-1} \dd \vf$, $\overline\kappa:=(\kappa_t,\kappa_x)$ and
 \begin{align}\label{lem:av_constants}
 \begin{split}
  &p:=\frac{s(m-1)+1-\gamma+\rho}{\rho\mu+(1-\rho)(s(m-1)+1-\gamma)}, \\
  \kappa_t:= &\frac{(1-s)(\mu-1+\rho)}{s(m-1)+1-\gamma+\rho},  
  \qquad \kappa_x:= \frac{2s(\mu-1+\rho)}{s(m-1)+1-\gamma+\rho}.
  \end{split}
 \end{align}
 Moreover, we have the estimate
 \begin{align}\label{lem:av_est1}
  &\norm{\overline{f}}_{S^{\overline{\kappa}}_{p,\infty,(\infty)}\dot B}\lesssim \||v|^{1-\gamma}g_{0}\|_{\calm_{TV}}+\||v|^{-\gamma}g_{1}\|_{\calm_{TV}} + \|f\|_{\LR{\rhon}_{t,x,v}}.
 \end{align}
 If additionally $\overline{f}\in \LR{r}_{t,x}$, $p\ne r\in[1,\infty]$, then for all $q\in\vpp{\min\{p,r\},\max\{p,r\}}$ it holds  $\overline{f}\in S^{\vt\overline{\kappa}}_{q,\infty}\dot B$, where $\vartheta\in\vpp{0,1}$ is such that
 \begin{align*}
  \frac{1}{q}=\frac{1-\vartheta}{r}+\frac{\vartheta}{p}.
 \end{align*}

 In this case we have
 \begin{align}\label{lem:av_est2}
  &\norm{\overline{f}}_{S^{\vt\overline{\kappa}}_{q,\infty}\dot B}\lesssim \||v|^{1-\gamma}g_{0}\|_{\calm_{TV}}+\||v|^{-\gamma}g_{1}\|_{\calm_{TV}} + \|f\|_{\LR{\rhon}_{t,x,v}} + \|\overline{f}\|_{\LR{r}_{t,x}}.
 \end{align}
 Finally, if $s=1$ and consequently $\kappa_t=0$, then \eqref{lem:av_est2} remains true if we replace the space $S^{\vt\overline{\kappa}}_{q,\infty}\dot B=S^{(0,\vt\kappa_x)}_{q,\infty}\dot B$ by $\td{L}^{q}_{t}\dot B^{\vt\kappa_x}_{q,\infty}$.
 \end{lem}
 \begin{rem}\label{rem:av}
 Observe that for $\rho\in(\frac{m+1-\gamma-\mu}{m+1-\gamma},1)$ one may prescribe a specific integrability exponent. More precisely, given \[\td p\in [\frac{1-\gamma+\rho}{\rho\mu+(1-\rho)(1-\gamma)},\frac{m+1-\gamma}{\mu}]\cap (1,\frac{m+1-\gamma}{\mu}]\] choose 
 \begin{align*}
  s:=\frac{\mu \td p \rho + \td p(1-\rho)(1-\gamma)-1+\gamma+\rho}{(m-1)(1-\td p(1-\rho))}\in (\frac{\mu-2+\gamma}{m-1},1]\cap [0,1].
 \end{align*}
 Then \eqref{lem:av_constants} reads $p=\td p$ as well as
 \begin{align*}
  \kappa_t&:= \frac{m+\rho-\gamma-\mu p \rho + p(1-\rho)(\gamma-m)}{p \rho}\frac{1}{m-1}, \\
  \kappa_x&:= \frac{\mu p \rho +p(1-\rho)(1-\gamma)-1+\gamma-\rho}{p \rho}\frac{2}{m-1}.
 \end{align*}
 Observe that in the limiting case $\rho\to 1$ and $\gamma\to 1$, these orders of differentiability correspond to the ones found in \eqref{scaling_const}.
\end{rem}

\begin{proof}[Proof of Lemma \ref{lem:av}]
We first assume that $f$ is compactly supported with respect to the variable $v$. This condition will enter only qualitatively, and never appears in quantitative form. Therefore, at the end of the proof, we can again remove this additional assumption. 

Since we are interested in regularity in terms of homogeneous Besov spaces, we decompose $f$ into Littlewood-Paley blocks with respect to the $t$-variable and the $x$-variable. Let $\{\eta_l\}_{l\in\Z}$ be a partition of unity on $\R\setminus\set{0}$ and $\{\vp_j\}_{j\in\Z}$ a partition of unity on $\R^d\setminus\set{0}$ as in Section \ref{FctSpc}. Then we define for $l,j\in\Z$
\begin{align*}
 f_{l,j}:=\calf^{-1}_{t,x}[\eta_l\vp_j\calf_{t,x}f],
\end{align*}
where $\calf_{t,x}{f}_{l,j}(\tau,\xi,v)$ is supported on frequencies $|\xi|\sim2^{j}$, $|\tau|\sim2^{l}$ for $l,j\in\Z$.
Similarly, we define the decompositions $g_{0,l,j}$ and $g_{1,l,j}$ of $g_0$ and $g_1$, respectively.
We consider a micro-local decomposition of $f_{l,j}$ connected to the degeneracy of the operator $\mcL(\partial_{t},\nabla_{x},v)$. Let $\psi_{0}\in\CRci(\R)$ be a smooth function supported in $B_{2}(0)$ and set $\psi_{1}:=1-\psi_0$.
For $\d>0$ to be specified later we write 
\begin{align*}
f_{l,j} & =\mcF_{x}^{-1}\psi_{0}\left(\frac{|\vf||\xi|^2}{\d}\right)\mcF_{x} f_{l,j}+\mcF_{x}^{-1}\psi_{1}\left(\frac{|\vf||\xi|^2}{\d}\right)\mcF_{x}f_{l,j} =:f_{l,j}^{0}+f_{l,j}^{1}.
\end{align*}
Since $f$ is a solution to \eqref{av_eqn}, we have
\begin{equation}\label{eq:eqn_fK-3}
\begin{split}
\mcF_{t,x}^{-1}&\mcL(i\tau,i\xi,v)\mcF_{t,x}f_{l,j}^{1}(t,x,v)\\
&=\mcF_{x}^{-1}\psi_{1}\left(\frac{|\vf||\xi|^2}{\d}\right)\mcF_{x}\left(g_{0,l,j}(t,x,v)+\partial_{v}g_{1,l,j}(t,x,v)\right)
\end{split}
\end{equation}
and thus
\begin{align}
f_{l,j}^{1}(t,x,v)= &  \mcF_{t,x}^{-1} \psi_{1}\left(\frac{|\vf||\xi|^2}{\d}\right)\frac{1}{\mcL(i\tau,i\xi,v)}\mcF_{t,x}g_{0,l,j}(t,x,v)\label{eq:eqn_fK-1-1}\\
 & + \mcF_{t,x}^{-1} \psi_{1}\left(\frac{|\vf||\xi|^2}{\d}\right)\frac{1}{\mcL(i\tau,i\xi,v)}\mcF_{t,x}\partial_{v}g_{1,l,j}(t,x,v)\nonumber \\
=: & f_{l,j}^{2}(t,x,v)+f_{l,j}^{3}(t,x,v).\nonumber 
\end{align}

In conclusion, we have arrived at the decomposition
\[
\begin{split}
\bar{f}_{l,j}&:=\int f_{l,j}|v|^{\mu-1}\,dv\\
& =\int f_{l,j}^{0}|v|^{\mu-1}\dd v+\int f_{l,j}^{2}|v|^{\mu-1}\dd v+\int f_{l,j}^{3}|v|^{\mu-1}\dd v=:\bar{f}_{l,j}^{0}+\bar{f}_{l,j}^{2}+\bar{f}_{l,j}^{3}.
\end{split}
\]
We aim to estimate the regularity of of these three contributions separately.

\smallskip\noindent
\textit{Step 1:} $f^0$. 
We note that we have the estimate $\|\calf^{-1}_t \eta_l \calf f\|_{\LR{\rhon}_{t,x}}\lesssim \|f\|_{\LR{\rhon}_{t,x}}$ with a constant independent of $l$, since $\|\eta_l\|_{\calm^\rhon}=\|\eta_0\|_{\calm^\rhon}<\infty$. Let $l,j\in\Z$ be arbitrary, fixed. Then, we have that $|v|\le 2\cdot 2^{-2j}\d$ on the support of $\vp(2^{-j}\xi)\psi_{0}\left(\frac{|\vf||\xi|^2}{\d}\right)$, so that $|\Omega_m(2^j,\delta)|\lesssim |[-2\cdot 2^{-2j}\d,2\cdot 2^{-2j}\d]|\lesssim 2^{-2j}\d$. Hence, by the isotropic truncation property and Minkowski's and H\"older's inequality it holds
\begin{align*}
& \|\int f_{l,j}^{0}|v|^{\mu-1}\dd v\|_{L_{t,x}^{\rhon}} \\
& \qquad = \|\int \mcF_{x}^{-1}\psi_{0}\left(\frac{|\vf||\xi|^2}{\d}\right) |v|^{\mu-1} \mcF_{x}f_{l,j}\dd v\|_{L_{t,x}^{\rhon}} \\
 & \qquad \lesssim \int \|\mcF_{x}^{-1}\psi_{0}\left(\frac{|\vf||\xi|^2}{\d}\right) |v|^{\mu-1} \mcF_{x}f_{l,j}\|_{L_{t,x}^{\rhon}}\dd v \\
 &\qquad \lesssim \left(\frac{\d}{2^{2j}}\right)^{\mu-1} \int\|\mcF_{x}^{-1}\psi_{0}\left(\frac{|\vf||\xi|^2}{\d}\right) \mcF_{x}f_{l,j}\|_{L_{t,x}^{\rhon}}\dd v \\
 &\qquad \lesssim \left(\frac{\d}{2^{2j}}\right)^{\mu-1} \int \|M_{\psi_0,2^{-j}}\|_{\calm^\rhon}\|f\|_{L_{t,x}^{\rhon}}\dd v \\
  &\qquad \le \left(\frac{\d}{2^{2j}}\right)^{\mu-1} \left\| \|M_{\psi_0,2^{-j}}\|_{\calm^\rhon}\right\|_{\LR{\rhon'}_{v}}\|f\|_{L_{t,x,v}^{\rhon}} \\
 &\qquad \lesssim \left(\frac{\d}{2^{2j}}\right)^{\mu-1} |\Omega_m(2^j,\delta)|^{\frac{1}{\rhon'}}\|f\|_{L_{t,x,v}^{\rhon}}
   \lesssim \left(\frac{\d}{2^{2j}}\right)^{\mu-1+\rho}\|f\|_{L_{t,x,v}^{\rhon}},
\end{align*}
where we have used $\beta'=\frac{1}{\rho}$.

\smallskip\noindent
\textit{Step 2:} $f^2$. 
Let $l,j\in\Z$ be arbitrary, fixed. Since $s\in [0,1]$, we clearly have
\begin{align*}
 |\tau|^{1-s}|v|^{s(m-1)}|\xi|^{2s}\le |\mcL(i\tau,i\xi,v)|.
\end{align*}
Moreover, in virtue of $s>\frac{\mu-2+\gamma}{m-1}$ we have on the support of $\eta_l\vp_j\psi_{1}\left(\frac{|\vf||\xi|^2}{\d}\right)$ (so that $|\tau|\sim 2^{l}$, $|\xi|\sim 2^{j}$, and $|v|\gtrsim 2^{-2j}\d$)
\begin{align*}
 \frac{|\vf|^{\mu-2+\gamma}}{|\mcL(i\tau,i\xi,v)|}&\lesssim \frac{|\vf|^{\mu-2+\gamma}}{|\tau|^{1-s}|v|^{s(m-1)}|\xi|^{2s}}\\
 & \lesssim \frac{\left(2^{-2j}\d\right)^{\mu-2+\gamma-s(m-1)}}{2^{l(1-s)}2^{2js}} =\frac{2^{2j(s(m-2)-\mu+2-\gamma)}}{\d^{s(m-1)-\mu+2-\gamma}2^{l(1-s)}}.
\end{align*}
Hence, by Theorem \ref{thm:FM} and Lemma \ref{lem:FM}, $\frac{|\vf|^{\mu-2+\gamma}}{\mcL(i\tau,i\xi,v)}$ acts on the support of $\eta_l\vp_j\psi_{1}\left(\frac{|\vf||\xi|^2}{\d}\right)$ as a constant multiplier of order $\frac{2^{2j(s(m-2)-\mu+2-\gamma)}}{\d^{s(m-1)-\mu+2-\gamma}2^{l(1-s)}}$. Consequently, by the anisotropic truncation property
\begin{align*}
& \|\int f_{l,j}^{2}|v|^{\mu-1}\dd v\|_{L_{t,x}^{1}}\\
& \qquad =  \|\int\mcF_{t,x}^{-1} \psi_{1}\left(\frac{|\vf||\xi|^2}{\d}\right)\frac{|\vf|^{\mu-2+\gamma}}{\mcL(i\tau,i\xi,v)}\mcF_{t,x}|v|^{1-\gamma}g_{0,l,j}\dd v\|_{L_{t,x}^{1}}\\
& \qquad \lesssim  \frac{2^{2j(s(m-2)-\mu+2-\gamma)}}{\d^{s(m-1)-\mu+2-\gamma}2^{l(1-s)}}\||v|^{1-\gamma}g_{0}\|_{\calm_{TV}}.
\end{align*}
Here, we have used that with $\psi_0\left(\frac{|\vf||\xi|^2}{\d}\right)$ also $\psi_1\left(\frac{|\vf||\xi|^2}{\d}\right)=1-\psi_0\left(\frac{|\vf||\xi|^2}{\d}\right)$ is a bounded $\calm_{TV}$-multiplier independent of $\d>0$.

\smallskip\noindent
\textit{Step 3:} $f^{3}$. 
Let $l,j\in\Z$ arbitrary, fixed.
We observe (recall $\mcL(i\tau,i\xi,v) = i\tau + |v|^{m-1}|\xi|^2$)
\begin{align*}
\int f_{l,j}^{3}&|v|^{\mu-1}\dd v \\
 &  =-\int\mcF_{t,x}^{-1} \psi_{1}'\left(\frac{|\vf||\xi|^2}{\d}\right)\frac{\sgn(v)|\xi|^2}{\d}\frac{|\vf|^{\mu-1}}{\mcL(i\tau,i\xi,v)}\mcF_{t,x}g_{1,l,j}\dd v \\
& -(\mu-1)\int\mcF_{t,x}^{-1} \psi_{1}\left(\frac{|\vf||\xi|^2}{\d}\right)\frac{\sgn(v)|\vf|^{\mu-2}}{\mcL(i\tau,i\xi,v)}\mcF_{t,x}g_{1,l,j}\dd v \\
& +\int\mcF_{t,x}^{-1} \psi_{1}\left(\frac{|\vf||\xi|^2}{\d}\right)\frac{|\vf|^{\mu-1} \partial_v\mcL(i\tau,i\xi,v)}{\mcL(i\tau,i\xi,v)^2}\mcF_{t,x}g_{1,l,j}\dd v \\
&=-\int\mcF_{t,x}^{-1} \psi_{1}'\left(\frac{|\vf||\xi|^2}{\d}\right)\frac{|\vf||\xi|^2}{\d}\frac{\sgn(v)|\vf|^{\mu-2+\gamma}}{\mcL(i\tau,i\xi,v)}\mcF_{t,x}|\vf|^{-\gamma}g_{1,l,j}\dd v \\
& -(\mu-1)\int\mcF_{t,x}^{-1} \psi_{1}\left(\frac{|\vf||\xi|^2}{\d}\right)\frac{\sgn(v)|\vf|^{\mu-2+\gamma}}{\mcL(i\tau,i\xi,v)}\mcF_{t,x}|\vf|^{-\gamma}g_{1,l,j}\dd v \\
& +(m-1)\int\mcF_{t,x}^{-1} \psi_{1}\left(\frac{|\vf||\xi|^2}{\d}\right)\frac{|\vf|^{\mu +m-3+\gamma}|\xi|^2}{\mcL(i\tau,i\xi,v)^2}\mcF_{t,x}|\vf|^{-\gamma}g_{1,l,j}\dd v.
\end{align*}
Observe that $\psi_1'$ is supported on an annulus. Therefore, we have as before $|\tau|\sim 2^{l}$, $|\xi|\sim 2^{j}$ and $|v|\gtrsim 2^{-2j}\d$ on the support of $\eta_l\vp_j\psi_{1}\left(\frac{|\vf||\xi|^2}{\d}\right)$, and additionally also $|v|\sim 2^{-2j}\d$ on the support of $\eta_l\vp_j\psi_{1}'\left(\frac{|\vf||\xi|^2}{\d}\right)$. This last observation allows us to estimate the expression $\frac{|\vf||\xi|^2}{\d}$ appearing in the first integral on the right hand side by
\begin{align*}
 \frac{|\vf||\xi|^2}{\d}\lesssim 1.
\end{align*}
As in Step 2, we obtain
\begin{align*}
 \frac{|\vf|^{\mu-2+\gamma}}{|\mcL(i\tau,i\xi,v)|}&\lesssim \frac{2^{2j(s(m-2)-\mu+2-\gamma)}}{\d^{s(m-1)-\mu+2-\gamma}2^{l(1-s)}},
\end{align*}
and, similarly,
\begin{align*}
 \frac{|\vf|^{\mu+m-3+\gamma}|\xi|^2}{|\mcL(i\tau,i\xi,v)|^2}&=\frac{|\vf|^{\mu-2+\gamma}}{|\mcL(i\tau,i\xi,v)|}\frac{|\vf|^{m-1}|\xi|^2}{|\mcL(i\tau,i\xi,v)|}\\
  & \lesssim \frac{|\vf|^{\mu-2+\gamma}}{|\mcL(i\tau,i\xi,v)|}\lesssim \frac{2^{2j(s(m-2)-\mu+2-\gamma)}}{\d^{s(m-1)-\mu+2-\gamma}2^{l(1-s)}}.
\end{align*}
In virtue of these estimates, the expressions
\begin{align*}
 \frac{|\vf||\xi|^2}{\d}\frac{\sgn(v)|\vf|^{\mu-2+\gamma}}{\mcL(i\tau,i\xi,v)}, \quad \frac{\sgn(v)|\vf|^{\mu-2+\gamma}}{\mcL(i\tau,i\xi,v)}, \quad \frac{|\vf|^{\mu +m-3+\gamma}|\xi|^2}{\mcL(i\tau,i\xi,v)^2}
\end{align*}
extend by Theorem \ref{thm:FM} and Lemma \ref{lem:FM} to constant multipliers of order $\frac{2^{2j(s(m-2)-\mu+2-\gamma)}}{\d^{s(m-1)-\mu+2-\gamma}2^{l(1-s)}}$ on supports of $\eta_l\vp_j\psi_{1}'\left(\frac{|\vf||\xi|^2}{\d}\right)$ and $\eta_l\vp_j\psi_{1}\left(\frac{|\vf||\xi|^2}{\d}\right)$, respectively. Hence, by the anisotropic truncation property, we obtain
\begin{align*}
\|\int f_{l,j}^{3}|v|^{\mu-1}\dd v\|_{L_{t,x}^{1}} & \lesssim \frac{2^{2j(s(m-2)-\mu+2-\gamma)}}{\d^{s(m-1)-\mu+2-\gamma}2^{l(1-s)}}\||v|^{-\gamma}g_{1,j}\|_{\calm_{TV}}.
\end{align*}

\smallskip\noindent
\textit{Step 4:} Conclusion. 
We aim to conclude by real interpolation. We set, for $\tauz>0$,
\begin{align*}
K(\tauz,\overline{f}_{l,j}):=\inf\{ & \|\overline{f}_{l,j}^{1}\|_{\LR{1}_{t,x}}+\tauz\|\overline{f}_{l,j}^{0}\|_{\LR{\rhon}_{t,x}}:\overline{f}_{l,j}^{0}\in \LR{\rhon}_{t,x}, \overline{f}_{l,j}^{1}\in \LR{1}_{t,x},\ \overline{f}_{l,j}=\overline{f}_{l,j}^{0}+\overline{f}_{l,j}^{1}\}.
\end{align*}
By the above estimates we obtain
\begin{align*}
K(\tauz,\overline{f}_{l,j}) & \lesssim \frac{2^{2j(s(m-2)-\mu+2-\gamma)}}{\d^{s(m-1)-\mu+2-\gamma}2^{l(1-s)}}(\||v|^{1-\gamma}g_{0}\|_{\calm_{TV}}+\||v|^{-\gamma}g_{1}\|_{\calm_{TV}})\\
 & \ \ +\tauz\left(\frac{\delta}{2^{2j}}\right)^{\mu-1+\rho}\|f\|_{\LR{\rhon}_{t,x,v}}.
\end{align*}
We now equilibrate the first and the second term on the right hand side, that is, we choose $\delta>0$ such that
\[
\frac{2^{2j(s(m-2)-\mu+2-\gamma)}}{\d^{s(m-1)-\mu+2-\gamma}2^{l(1-s)}}=\tauz\left(\frac{\delta}{2^{2j}}\right)^{\mu-1+\rho},
\]
that is,
\begin{align*}
 \d^{-a}c^{1-s}d^{-a+s}=\tauz\d^{b}d^b
\end{align*}
with $a:=s(m-1)-\mu+2-\gamma$, $b:=\mu-1+\rho$, $c:=2^{-l}$ and $d:=2^{-2j}$.
This yields
\[
\d=\tauz^{-\frac{1}{a+b}}c^{\frac{1-s}{a+b}}d^{\frac{s-a-b}{a+b}},
\]
and further
\[
\d^{-a}c^{1-s}d^{-a+s}=\tauz^{\frac{a}{a+b}}c^{\frac{(1-s)b}{a+b}}d^{\frac{sb}{a+b}}.
\]
Hence, with 
\[
\t:=\frac{a}{a+b}=\frac{s(m-1)-\mu+2-\gamma}{s(m-1)+1-\gamma+\rho}
\]
we obtain
\begin{align*}
\lefteqn{\tauz^{-\t}K(\tauz,\overline{f}_{l,j})}\\
 && \lesssim  2^{-l\frac{(1-s)(\mu-1+\rho)}{s(m-1)+1-\gamma+\rho}}2^{-2j\frac{s(\mu-1+\rho)}{s(m-1)+1-\gamma+\rho}}(\||v|^{1-\gamma}g_{0}\|_{\calm_{TV}}+\||v|^{-\gamma}g_{1}\|_{\calm_{TV}} + \|f\|_{\LR{\rhon}_{t,x,v}})\\
&&= 2^{-l\kappa_t}2^{-j\kappa_x}(\||v|^{1-\gamma}g_{0}\|_{\calm_{TV}}+\||v|^{-\gamma}g_{1}\|_{\calm_{TV}} + \|f\|_{\LR{\rhon}_{t,x,v}}).
\end{align*}
Observe that $1-\t + \frac{\t}{\rhon}=1-\t + \t(1-\rho)=1-\t\rho$, so that $(\LR{1}_{t,x},\LR{\rhon}_{t,x})_{\t,\infty}=\LR{p,\infty}_{t,x}$ with $p=\frac{1}{1-\t\rho}=\frac{a+b}{a(1-\rho)+b}=\frac{s(m-1)+1-\gamma+\rho}{\rho\mu+(1-\rho)(s(m-1)+1-\gamma)}$. Hence, we may take the supremum over $\tauz>0$ to obtain
\begin{equation}\label{av_est_lj}
\|\overline{f}_{l,j}\|_{\LR{p,\infty}_{t,x}}  \lesssim 2^{-l\kappa_t}2^{-j\kappa_x}(\||v|^{1-\gamma}g_{0}\|_{\calm_{TV}}+\||v|^{-\gamma}g_{1}\|_{\calm_{TV}} + \|f\|_{\LR{\rhon}_{t,x,v}}).
\end{equation}
Multiplying by $2^{l\kappa_t}2^{j\kappa_x}$ and taking the supremum over $j,l\in \Z$ yields \eqref{lem:av_est1}.

If we assume additionally $\overline{f}\in \LR{r}_{t,x}$, $r\neq p$, we choose for $q\in\vpp{\min\{p,r\},\max\{p,r\}}$ a corresponding $\vartheta\in\vpp{0,1}$ subject to $1/q=(1-\vartheta)/r+\vartheta/p$. Then using $(\LR{r}_{t,x},\LR{p,\infty}_{t,x})_{\vartheta,q}=\LR{q}_{t,x}$, together with \eqref{av_est_lj}, we obtain
\begin{align*}
\|\overline{f}_{l,j}\|_{\LR{q}_{t,x}} &\lesssim \|\overline{f}_{l,j}\|_{\LR{r}_{t,x}}^{1-\vartheta}\|\overline{f}_{l,j}\|_{\LR{p,\infty}_{t,x}}^\vartheta \nonumber\\
& \lesssim \|\overline{f}\|_{\LR{r}_{t,x}}^{1-\vartheta}2^{-l\vartheta\kappa_t}2^{-j\vartheta\kappa_x}(\||v|^{1-\gamma}g_{0}\|_{\calm_{TV}}+\||v|^{-\gamma}g_{1}\|_{\calm_{TV}} + \|f\|_{\LR{\infty}_{t,x,v}})^{\vartheta} \nonumber \\
&\le 2^{-l\vartheta\kappa_t}2^{-j\vartheta\kappa_x}(\||v|^{1-\gamma}g_{0}\|_{\calm_{TV}}+\||v|^{-\gamma}g_{1}\|_{\calm_{TV}} + \|f\|_{\LR{\rhon}_{t,x,v}} + \|\overline{f}\|_{\LR{r}_{t,x}}).
\end{align*}
Multiplying by $2^{l\vartheta\kappa_t}2^{j\vartheta\kappa_x}$ and taking the supremum over $j,l\in \Z$ yields \eqref{lem:av_est2}.

Finally we note that if $s=1$ and consequently $\kappa_t=0$, then the partition of unity $\{\eta_l\}_{l\in\Z}$ in the Fourier space connected to the time variable $t$ is not necessary. Hence, if we set $\alpha_\tau=0$ whenever Lemma \ref{lem:FM} is invoked and replace Theorem \ref{thm:FM} by its isotropic variant (cf.\@ Remark \ref{rem:FM}), we obtain
%
\begin{align*}
\|\overline{f}_{j}\|_{\LR{q}_{t,x}} &\lesssim 2^{-j\vartheta\kappa_x}(\||v|^{1-\gamma}g_{0}\|_{\calm_{TV}}+\||v|^{-\gamma}g_{1}\|_{\calm_{TV}} + \|f\|_{\LR{\rhon}_{t,x,v}} + \|\overline{f}\|_{\LR{r}_{t,x}}),
\end{align*}
which shows $\overline{f}\in \td{L}^q\dot B^{\vartheta\kappa_x}_{q,\infty}$.

It remains to consider the case when $f$ is not localized in $v$. We observe that for a smooth cut-off function $\psi\in \CRci(\R)$, the function $(t,x,v)\to f(t,x,v)\psi(v)=:f^\psi(t,x,v)$ is a solution to
\begin{align*}
  \diffop(\partial_t,\nabla_x,\vf) f^\psi(t,x,\vf)=\g_0^\psi(t,x,\vf)+g_1^{\psi'}(t,x,v)+\partial_\vf \g_1^\psi(t,x,\vf) \ \ton \R_t \times \R^d_x \times \R_\vf,
 \end{align*}
 where $\g_0^\psi$, $g_1^{\psi'}$ and $\g_1^\psi$ are defined analogously.
 Hence, estimate \eqref{av_est_lj} reads in this case
 \begin{align*}
\|\overline{f^\psi_{l,j}}\|_{\LR{p,\infty}_{t,x}} \le 2^{-l\vartheta\kappa_t}2^{-j\vartheta\kappa_x}(\||v|^{1-\gamma} (g_{0}^\psi+g_{1}^{\psi'})\|_{\calm_{TV}}+\||v|^{-\gamma}g_{1}^\psi\|_{\calm_{TV}} + \|f^\psi\|_{\LR{\rhon}_{t,x,v}}).\end{align*}
 Since $|v|^{-\gamma}g_1\in \calm_{TV}$ by assumption, there exists to $\eps_n\downarrow 0$ a sequence $r_n\uparrow\infty$ such that 
 \[
 \int_{\R_t\times \R^d_x\times \R_v} \chi_{\set{r_n\le |v|}} |v|^{-\gamma} g_1 \dd v \dd x \dd t\le \eps_n\]
  for all $n\in \N$. For $n\in\N$ and a smooth cut-off function $\psi\in \CRci(\R)$ with $\psi=1$ on $B_1(0)$ and $\supp\psi\in B_2(0)$, we define $\psi_n$ via $\psi_n(v):=\psi(v/n)$. Hence $\psi'_n$ is supported on $r_n\le |v|\le 2r_n$ and takes values in $[0,1/r_n]$, so that we may estimate
 \begin{align*}
  \||v|^{1-\gamma} g_{1}^{\psi_n'}\|_{\calm_{TV}}&=\int_{\R_t\times \R^d_x\times \R_v} |\psi'_n(v)||v|(|v|^{-\gamma}g_1)\dd v \dd x \dd t \\
  &=\int_{\R_t\times \R^d_x\times \R_v} \chi_{\set{r_n\le |v|\le 2 r_n}}|\psi'_n(v)||v|(|v|^{-\gamma}g_1)\dd v \dd x \dd t \\ &\lesssim \int \chi_{r_n\le |v|\le 2 r_n} |v|^{-\gamma} g_1 \dd v\le \eps_n.
 \end{align*}
 Thus, taking the limit $n\to\infty$ and using Fatou's lemma, we obtain \eqref{av_est_lj} also for general $f$. Multiplying by $2^{l\vartheta\kappa_t}2^{j\vartheta\kappa_x}$ and taking the supremum over $j,l\in \Z$, we may conclude as before.
\end{proof}
%
%
\begin{lem}\label{lem:av2}
 Assume $\gamma\in(-\infty,1)$, $m\in \vpp{1,\infty}$, $\mu\in [1,2-\gamma)$, $\rho\in(0,1]$, $\beta'=\frac{1}{\rho}$, and let $f$, $g_0$, $g_1$, and $\overline{f}$ be as in Lemma \ref{lem:av}. Define 
 \begin{align}\label{lem:av2_constants}
 \begin{split}
  p&:=\frac{1-\gamma+\rho}{\rho\mu+(1-\rho)(1-\gamma)}, \qquad \kappa_t:= \frac{\mu-1+\rho}{1-\gamma+\rho}. 
  \end{split}
 \end{align}
 If $\overline{f}\in \LR{r}_{t,x}$, $p\ne r\in[1,\infty]$, then for all $q\in\vpp{\min\{p,r\},\max\{p,r\}}$ we have $\overline{f}\in \td{L}^{q}_x\dot B^{\vartheta\kappa_t}_{q,\infty}$, where $\vartheta\in\vpp{0,1}$ is such that
 \begin{align*}
  \frac{1}{q}=\frac{1-\vartheta}{r}+\frac{\vartheta}{p}.
 \end{align*}
 Moreover,
 \begin{align}\label{lem:av2_est}
  &\norm{\overline{f}}_{\td{L}^{q}_x\dot B^{\vartheta\kappa_t}_{q,\infty}}\lesssim \||v|^{1-\gamma}g_{0}\|_{\calm_{TV}}+\||v|^{-\gamma}g_{1}\|_{\calm_{TV}} + \|f\|_{\LR{\rhon}_{t,x,v}} + \|\overline{f}\|_{\LR{r}_{t,x}}.
 \end{align}
 \end{lem}
 
 \begin{proof}
By the same arguments as in the proof of Lemma \ref{lem:av}, we may assume that $f$ is localized in $v$. In fact, the whole proof of Lemma \ref{lem:av2} is similar to the one of Lemma \ref{lem:av}, with the modification that here we consider a micro-local decomposition of $f$ depending on the size of $v$ only and do not localize in the Fourier space connected to the spatial variable $x$. More precisely, let $\{\eta_l\}_{l\in\Z}$ be a partition of unity on $\R\setminus\set{0}$ as in Section \ref{FctSpc}. Then we define for $l\in\Z$
\begin{align*}
 f_{l}:=\calf^{-1}_{x}[\eta_l\calf_{t}f],
\end{align*}
where $\calf_{t}{f}_{l}(\tau,x,v)$ is supported on frequencies $|\tau|\sim2^{l}$ for $l\in\Z$.
Similarly, we define the decompositions $g_{0,l}$ and $g_{1,l}$ of $g_0$ and $g_1$, respectively. Moreover, we again consider a smooth function $\psi_{0}\in\CRci(\R)$ supported in $B_{2}(0)$ and set $\psi_{1}:=1-\psi_0$.
For $\d>0$ to be specified later we write 
\begin{align*}
f_l & =\psi_{0}\left(\frac{|\vf|}{\d}\right) f_l+\psi_{1}\left(\frac{|\vf|}{\d}\right) f_l =:f_l^{0}+f_l^{1}.
\end{align*}
Since $f$ is a solution to \eqref{av_eqn}, we have
\begin{align*}
\mcF_{t,x}^{-1}\mcL(i\tau,i\xi,v)\mcF_{t,x}f_l^{1}(t,x,v)=\psi_{1}\left(\frac{|\vf|}{\d}\right)\left(g_{0,l}(t,x,v)+\partial_{v}g_{1,l}(t,x,v)\right)
\end{align*}
and thus
\begin{align*}
f_l^{1}(t,x,v)= &  \mcF_{t,x}^{-1} \psi_{1}\left(\frac{|\vf|}{\d}\right)\frac{1}{\mcL(i\tau,i\xi,v)}\mcF_{t,x}g_{0,l}(t,x,v)\\
 & + \mcF_{t,x}^{-1} \psi_{1}\left(\frac{|\vf|}{\d}\right)\frac{1}{\mcL(i\tau,i\xi,v)}\mcF_{t,x}\partial_{v}g_{1,l}(t,x,v)\nonumber \\
=: & f_l^{2}(t,x,v)+f_l^{3}(t,x,v),\nonumber 
\end{align*}
so that we arrive at the decomposition
\begin{align*}
\bar{f}_l &:=\int f_l|v|^{\mu-1}\dd v=\int f_l^{0}|v|^{\mu-1}\dd v+\int f_l^{2}|v|^{\mu-1}\dd v+\int f_l^{3}|v|^{\mu-1}\dd v\\
& =:\bar{f}_l^{0}+\bar{f}_l^{2}+\bar{f}_l^{3}.
\end{align*}
Again, we treat the three contributions separately.

\smallskip\noindent
\textit{Step 1:} $f^0$. 
Let $l\in\Z$ be arbitrary, fixed. Since $|v|\lesssim \d$ on the support of $\psi_{0}\left(\frac{|\vf|}{\d}\right)$, using Minkowski's and H\"older's inequality, we have
\begin{align*}
\|\int f_{l}^{0}|v|^{\mu-1}\dd v\|_{L_{t,x}^{\rhon}} & = \|\int \psi_{0}\left(\frac{|\vf|}{\d}\right) |v|^{\mu-1} f_{l}\dd v\|_{L_{t,x}^{\rhon}} \\
 & \le \int |\psi_{0}|\left(\frac{|\vf|}{\d}\right) |v|^{\mu-1} \|f_{l}\|_{L_{t,x}^{\rhon}}\dd v \\
 &\lesssim \d^{\mu-1} \int|\psi_{0}|\left(\frac{|\vf|}{\d}\right) \|f_{l}\|_{L_{t,x}^{\rhon}}\dd v \\
 &\lesssim \d^{\mu-1} \|f\|_{L_{t,x,v}^{\rhon}}\left(\int |\psi_0|\left(\frac{|\vf|}{\d}\right)^{\rhon'}\dd v\right)^{\frac{1}{\rhon'}} \\
 & \lesssim \d^{\mu-1+\rho}\|f\|_{L_{t,x,v}^{\rhon}}.
\end{align*}

\smallskip\noindent
\textit{Step 2:} $f^2$. 
Let $l\in\Z$ be arbitrary, fixed.
Since $\mu\le 2-\gamma$, we have on the support of $\eta_l\psi_{1}\left(\frac{|\vf|}{\d}\right)$ (so that $|\tau|\sim 2^{l}$ and $|v|\ge \d$)
\begin{align*}
 \frac{|\vf|^{\mu-2+\gamma}}{|\mcL(i\tau,i\xi,v)|}&\lesssim \frac{|\vf|^{\mu-2+\gamma}}{|\tau|}\lesssim \frac{\d^{\mu-2+\gamma}}{2^l}.
\end{align*}
By Lemma \ref{lem:FM} applied with $\alpha_\xi=0$ and the isotropic variant of Theorem \ref{thm:FM} (cf.\@ Remark \ref{rem:FM}), $\frac{|\vf|^{\mu-2+\gamma}}{|\mcL(i\tau,i\xi,v)|}$ acts as a constant multiplier of order $\frac{\d^{\mu-2+\gamma}}{2^l}$ on the support of $\eta_l\psi_{1}\left(\frac{|\vf|}{\d}\right)$. Consequently
\begin{align*}
\|\int f_{l}^{2}|v|^{\mu-1}\dd v\|_{L_{t,x}^{1}} & =\|\int\mcF_{t,x}^{-1} \psi_{1}\left(\frac{|\vf|}{\d}\right)\frac{|\vf|^{\mu-2+\gamma}}{\mcL(i\tau,i\xi,v)}\mcF_{t,x}|v|^{1-\gamma}g_{0,l}\dd v\|_{L_{t,x}^{1}}\\
& \lesssim  \frac{\d^{\mu-2+\gamma}}{2^l}\||v|^{1-\gamma}g_{0}\|_{\calm_{TV}}.
\end{align*}

\smallskip\noindent
\textit{Step 3:} $f^{3}$. 
Let $l\in\Z$ arbitrary, fixed.
We observe (recall $\mcL(i\tau,i\xi,v) = i\tau + |v|^{m-1}|\xi|^2$)
\begin{align*}
\int f_{l}^{3}|v|^{\mu-1}\dd v & =-\int\mcF_{t,x}^{-1} \psi_{1}'\left(\frac{|\vf|}{\d}\right)\frac{\sgn(v)}{\d}\frac{|\vf|^{\mu-1}}{\mcL(i\tau,i\xi,v)}\mcF_{t,x}g_{1,l}\dd v \\
& -(\mu-1)\int\mcF_{t,x}^{-1} \psi_{1}\left(\frac{|\vf|}{\d}\right)\frac{\sgn(v)|\vf|^{\mu-2}}{\mcL(i\tau,i\xi,v)}\mcF_{t,x}g_{1,l}\dd v \\
& +\int\mcF_{t,x}^{-1} \psi_{1}\left(\frac{|\vf|}{\d}\right)\frac{|\vf|^{\mu-1} \partial_v\mcL(i\tau,i\xi,v)}{\mcL(i\tau,i\xi,v)^2}\mcF_{t,x}g_{1,l}\dd v \\
&=-\int\mcF_{t,x}^{-1} \psi_{1}'\left(\frac{|\vf|}{\d}\right)\frac{|\vf|}{\d}\frac{\sgn(v)|\vf|^{\mu-2+\gamma}}{\mcL(i\tau,i\xi,v)}\mcF_{t,x}|\vf|^{-\gamma}g_{1,l}\dd v \\
& -(\mu-1)\int\mcF_{t,x}^{-1} \psi_{1}\left(\frac{|\vf|}{\d}\right)\frac{\sgn(v)|\vf|^{\mu-2+\gamma}}{\mcL(i\tau,i\xi,v)}\mcF_{t,x}|\vf|^{-\gamma}g_{1,l}\dd v \\
& +(m-1)\int\mcF_{t,x}^{-1} \psi_{1}\left(\frac{|\vf|}{\d}\right)\frac{|\vf|^{\mu +m-3+\gamma}|\xi|^2}{\mcL(i\tau,i\xi,v)^2}\mcF_{t,x}|\vf|^{-\gamma}g_{1,l}\dd v
\end{align*}
Observe that $\psi_1'$ is supported on an annulus. Therefore, we have as before $|\tau|\sim 2^{l}$ and $|v|\ge \d$ on the support of $\eta_l\psi_{1}\left(\frac{|\vf|}{\d}\right)$, and additionally also $|v|\sim \d$ on the support of $\eta_l\psi_{1}'\left(\frac{|\vf|}{\d}\right)$. This last observation allows us to estimate the expression $\frac{|\vf|}{\d}$ appearing in the first integral on the right hand side by $\frac{|\vf|}{\d}\lesssim 1$. As in Step 2, we obtain
\begin{align*}
 \frac{|\vf|^{\mu-2+\gamma}}{|\mcL(i\tau,i\xi,v)|}&\lesssim \frac{\d^{\mu-2+\gamma}}{2^l},
\end{align*}
and, similarly,
\begin{align*}
 \frac{|\vf|^{\mu+m-3+\gamma}|\xi|^2}{|\mcL(i\tau,i\xi,v)|^2}&=\frac{|\vf|^{\mu-2+\gamma}}{|\mcL(i\tau,i\xi,v)|}\frac{|\vf|^{m-1}|\xi|^2}{|\mcL(i\tau,i\xi,v)|}\lesssim \frac{|\vf|^{\mu-2+\gamma}}{|\mcL(i\tau,i\xi,v)|}\lesssim \frac{\d^{\mu-2+\gamma}}{2^l}.
\end{align*}
In virtue of these estimates, Lemma \ref{lem:FM} applied with $\alpha_\xi=0$ and the isotropic variant of Theorem \ref{thm:FM} (cf.\@ Remark \ref{rem:FM}) show that the expressions
\begin{align*}
 \frac{|\vf|}{\d}\frac{\sgn(v)|\vf|^{\mu-2+\gamma}}{\mcL(i\tau,i\xi,v)}, \quad \frac{\sgn(v)|\vf|^{\mu-2+\gamma}}{\mcL(i\tau,i\xi,v)}, \quad \frac{|\vf|^{\mu +m-3+\gamma}|\xi|^2}{\mcL(i\tau,i\xi,v)^2}
\end{align*}
extend to constant multipliers of order $\frac{\d^{\mu-2+\gamma}}{2^l}$ on the supports of $\eta_l\psi_{1}'\left(\frac{|\vf|}{\d}\right)$ and $\eta_l\psi_{1}\left(\frac{|\vf|}{\d}\right)$, respectively. Hence, we obtain
\begin{align*}
\|\int f_{l}^{3}|v|^{\mu-1}\dd v\|_{L_{t,x}^{1}} & \lesssim \frac{\d^{\mu-2+\gamma}}{2^l}\||v|^{-\gamma}g_{1,j}\|_{\calm_{TV}}.
\end{align*}

\smallskip\noindent
\textit{Step 4:} Conclusion.
We aim to conclude by real interpolation. We set, for $\tauz>0$,
\begin{align*}
K(\tauz,\overline{f}_{l}):=\inf\{ & \|\overline{f}_{l}^{1}\|_{\LR{1}_{t,x}}+\tauz\|\overline{f}_{l}^{0}\|_{\LR{\rhon}_{t,x}}:\overline{f}_{l}^{0}\in \LR{\rhon}_{t,x}, \overline{f}_{l}^{1}\in \LR{1}_{t,x},\ \overline{f}_{l}=\overline{f}_{l}^{0}+\overline{f}_{l}^{1}\}.
\end{align*}
By the above estimates we obtain
\begin{align*}
K(\tauz,\overline{f}_{l}) & \lesssim \frac{\d^{\mu-2+\gamma}}{2^l}(\||v|^{1-\gamma}g_{0}\|_{\calm_{TV}}+\||v|^{-\gamma}g_{1}\|_{\calm_{TV}})+\tauz\d^{\mu-1+\rho}\|f\|_{\LR{\rhon}_{t,x,v}}.
\end{align*}
We now equilibrate the first and the second term on the right hand side, that is, we choose $\d>0$ such that
\[
\frac{\d^{\mu-2+\gamma}}{2^l}=\tauz\d^{\mu-1+\rho},
\]
that is,
\[
\d:=\tauz^{-\frac{1}{1-\gamma+\rho}}2^{-\frac{l}{1-\gamma+\rho}}.
\]
Hence, with 
$
\t:=\frac{-\mu+2-\gamma}{1-\gamma+\rho}
$
we obtain
\begin{align*}
\tauz^{-\t}K(\tauz,\overline{f}_{l}) & \lesssim 2^{-l\frac{\mu-1+\rho}{1-\gamma+\rho}}(\||v|^{1-\gamma}g_{0}\|_{\calm_{TV}}+\||v|^{-\gamma}g_{1}\|_{\calm_{TV}} + \|f\|_{\LR{\rhon}_{t,x,v}})\\
&= 2^{-l\kappa_t}(\||v|^{1-\gamma}g_{0}\|_{\calm_{TV}}+\||v|^{-\gamma}g_{1}\|_{\calm_{TV}} + \|f\|_{\LR{\rhon}_{t,x,v}}).
\end{align*}
As in Step 4 of the proof of Lemma \ref{lem:av} we use $(\LR{1}_{t,x},\LR{\rhon}_{t,x})_{\t,\infty}=\LR{p,\infty}_{t,x}$ with $p=\frac{1}{1-\t\rho}=\frac{1-\gamma+\rho}{\rho\mu+(1-\rho)(1-\gamma)}$ to obtain
\begin{align}\label{av2_est_lj}
\|\overline{f}_{l}\|_{\LR{p,\infty}_{t,x}} & \lesssim 2^{-l\kappa_t}(\||v|^{1-\gamma}g_{0}\|_{\calm_{TV}}+\||v|^{-\gamma}g_{1}\|_{\calm_{TV}} + \|f\|_{\LR{\rhon}_{t,x,v}}).
\end{align}
For $q\in\vpp{\min\{p,r\},\max\{p,r\}}$ we choose a corresponding $\vartheta\in\vpp{0,1}$ subject to $1/q=(1-\vartheta)/r+\vartheta/p$. Then using $(\LR{r}_{t,x},\LR{p,\infty}_{t,x})_{\vartheta,q}=\LR{q}_{t,x}$, together with \eqref{av2_est_lj}, we obtain
\begin{align*}
\|\overline{f}_{l}\|_{\LR{q}_{t,x}} &\lesssim \|\overline{f}_{l}\|_{\LR{r}_{t,x}}^{1-\vartheta}\|\overline{f}_{l}\|_{\LR{p,\infty}_{t,x}}^\vartheta \\
& \lesssim \|\overline{f}\|_{\LR{r}_{t,x}}^{1-\vartheta}2^{-l\vartheta\kappa_t}(\||v|^{1-\gamma}g_{0}\|_{\calm_{TV}}+\||v|^{-\gamma}g_{1}\|_{\calm_{TV}} + \|f\|_{\LR{\infty}_{t,x,v}})^{\vartheta} \\
&\le 2^{-l\vartheta\kappa_t}(\||v|^{1-\gamma}g_{0}\|_{\calm_{TV}}+\||v|^{-\gamma}g_{1}\|_{\calm_{TV}} + \|f\|_{\LR{\infty}_{t,x,v}} + \|\overline{f}\|_{\LR{r}_{t,x}}).
\end{align*}
Multiplying by $2^{l\vartheta\kappa_t}$ and taking the supremum over $l\in \Z$ yields \eqref{lem:av2_est}.
\end{proof}
%
\begin{cor}\label{cor:av0}
 Let $m\in\vpp{1,\infty}$, $\gamma\in\vpp{-\infty,m}$, $\mu\in[1,m+1-\gamma)$, $f\in \LR{1}_{t,x,v}\cap \LR{\infty}_{t,x,v}$ be a solution to \eqref{av_eqn}, and let $g_0$, $g_1$ and $\overline{f}$ be as in Lemma \ref{lem:av}. Let $q\in (1,\frac{m+1-\gamma}{\mu})$ and define
 \begin{align*}
  \td\kappa_x:= \frac{\mu q-1}{q}\frac{2}{m-\gamma}.
 \end{align*}
 If $\overline{f}\in \LR{1}(\R^{d+1})\cap \LR{q}(\R;\LR{1}(\R^d))$, then $\overline{f}\in \LR{q}(\R;\WSR{\sigma_x}{q}(\R^d))$ for all $\sigma_x\in[0,\td\kappa_x)$. Furthermore,
 \begin{align}\label{cor:av0_main_est}
  \norm{\overline{f}}_{\LR{q}_t(\WSR{\sigma_x}{q}_x)}\lesssim \||v|^{1-\gamma}g_{0}\|_{\calm_{TV}}+\||v|^{-\gamma}g_{1}\|_{\calm_{TV}} + \|f\|_{\LR{1}_{t,x,v} \cap \LR{\infty}_{t,x,v}} +  \|\overline{f}\|_{\LR{1}_{t,x}\cap \LR{q}_t\LR{1}_x}.
 \end{align}
\end{cor}
\begin{proof}
 We recall the decomposition $f_j=\calf_{x}^{-1}\vp_j\calf_{x}f$ introduced in the proof of Lemma \ref{lem:av}. We argue that it suffices to consider the case when $f_{j}=0$ for all $j<0$. Indeed, the part $f_{<}:=\sum_{j<0}f_j$ can be estimated in view of Bernstein's Lemma (cf.\@ \cite[Lemma 2.1]{BCD11}) via
 \begin{align*}
  \|\overline{f}_{<}\|_{\LR{q}_t(\WSR{\sigma_x}{q}_x)}\lesssim \|\overline{f}\|_{\LR{q}_{t}\LR{1}_x}.
 \end{align*}
 We aim to control $\overline{f}$ in $\td{L}^{q}_t\dot B^{\vt\kappa_x}_{q,\infty}$ where $\vt\in(0,1)$ is sufficiently large such that $\sigma_x<\vt\kappa_x$, and then use Lemma \ref{lem:emb_0} to the effect of
 \begin{align*}
  \|\overline{f}\|_{\LR{q}_t(\WSR{\sigma_x}{q}_x)}\lesssim \|\overline{f}\|_{\td{L}^{q}_t B^{\vt\kappa_x}_{q,\infty}} = \|\overline{f}\|_{\td{L}^{q}_t \dot B^{\vt\kappa_x}_{q,\infty}},
 \end{align*}
 where the last equality is apparent from the definition of the homogeneous and non-homogeneous Lebesgue-Besov spaces and the fact that the low frequencies of $f$ vanish.
 Thus, it remains to establish
 \begin{align}\label{cor:av0_est2}
  \norm{\overline{f}}_{\td{L}^{q}\dot B^{\vt\kappa_x}_{q,\infty}}\lesssim \||v|^{1-\gamma}g_{0}\|_{\calm_{TV}}+\||v|^{-\gamma}g_{1}\|_{\calm_{TV}} + \|f\|_{\LR{1}_{t,x,v}\cap \LR{\infty}_{t,x,v}}  + \|\overline{f}\|_{\LR{1}_{t,x}}.
 \end{align}
 For $\td p\in (1,\frac{m+1-\gamma}{\mu})$, choose
 \begin{align*}
  \rho:=\frac{(\td p-1)(m-\gamma)}{1+\td p(m-\mu-\gamma)}. 
 \end{align*}
 We claim that $\rho$ is positive and well-defined: Since the nominator is positive due to $\td p>1$ and $m>\gamma$, it remains to check that the denominator is positive. This is obvious for $\mu\le m-\gamma$. For $\mu> m-\gamma$, we observe that due to $\mu<m+1-\gamma$ we have
 \begin{align*}
  \td p<\frac{m+1-\gamma}{\mu}<\frac{1}{\mu+\gamma-m},
 \end{align*}
 which implies $1+\td p(m-\mu-\gamma)>0$. Moreover, $\td p<\frac{m+1-\gamma}{\mu}$ can be rewritten as $(\td p-1)(m-\gamma)<1+\td p(m-\mu-\gamma)$, so that $\rho\in(0,1)$. Hence, we may apply Lemma \ref{lem:av} with this choice of $\rho$ and with $s=1$. One checks that in this case the integrability and differentiability exponents in \eqref{lem:av_constants} read $p=\td p$, $\kappa_t=0$, and $\kappa_x=\frac{\mu \td p-1}{\td p}\frac{2}{m-\gamma}$.
 
 Choose $\td p\in(q, \frac{m+1-\gamma}{\mu})$ so that $\td\kappa_x<\kappa_x$ and define $\vt\in(0,1)$ through
 \begin{align*}
  \frac{1}{q}=1-\vt+\frac{\vt}{\td p}.
 \end{align*}
 We may choose $\td p\in(q, \frac{m+1-\gamma}{\mu})$ sufficiently small so that $\vt\in(0,1)$ is so large that $\sigma_x<\vt\td\kappa_x<\vt\kappa_x$. In view of \eqref{lem:av_est2} (with the space $S^{\vt\overline{\kappa}}_{q,\infty}\dot B=S^{(0,\vt\kappa_x)}_{q,\infty}\dot B$ replaced by $\td{L}^{q}_{t}\dot B^{\vt\kappa_x}_{q,\infty}$) we obtain
 \begin{align*}
  \|\overline{f}_{j}\|_{\LR{q}_{t,x}}\lesssim 2^{-j\vt\kappa_x}(\||v|^{1-\gamma}g_{0}\|_{\calm_{TV}}+\||v|^{-\gamma}g_{1}\|_{\calm_{TV}} + \|f\|_{\LR{\rhon}_{t,x,v}} + \|\overline{f}\|_{\LR{1}_{t,x}}),
 \end{align*}
 where we recall the notation $\overline{f}_{j}:=\int \calf^{-1}_{x}[\vp_j \calf_{x}f]|v|^{\mu-1} \dd v$.
 If we multiply by $2^{j\vt\kappa_x}$ and take the supremum over $j\in \Z$, this yields
 \begin{align*}
  \norm{\overline{f}}_{\td L^{q}_t\dot B^{\vt\kappa_x}_{q,\infty}}\lesssim \||v|^{1-\gamma}g_{0}\|_{\calm_{TV}}+\||v|^{-\gamma}g_{1}\|_{\calm_{TV}} + \|f\|_{\LR{\rhon}_{t,x,v}} + \|\overline{f}\|_{\LR{1}_{t,x}}.
 \end{align*}
 By the estimate $\|f\|_{\LR{\rhon}_{t,x,v}} \lesssim \|f\|_{\LR{1}_{t,x,v}} + \|f\|_{\LR{\infty}_{t,x,v}}$, this gives \eqref{cor:av0_est2}.
\end{proof}


\begin{cor}\label{cor:av}
 Let $m\in\vpp{1,\infty}$, $\gamma\in\vpp{-\infty,1}$, $f\in \LR{1}_{t,x,v}\cap \LR{\infty}_{t,x,v}$ be a solution to \eqref{av_eqn}, and let $g_0$ and $g_1$ be as in Lemma \ref{lem:av}. Assume $\overline{f}\in \LR{r}_{t,x}$ for all $r\in [1,m+1-\gamma)$, where $\overline{f}(t,x):=\int f(t,x,v)\dd v$. Let $\td p\in (2-\gamma,m+1-\gamma)$ and define
 \begin{align*}
  \td\kappa_t:= \frac{m+1-\gamma-\td p}{\td p}\frac{1}{m-1}, \qquad
  \td\kappa_x:= \frac{\td p-2+\gamma}{\td p}\frac{2}{m-1}.
 \end{align*}
 Then $\overline{f}\in \WSR{\sigma_t}{\td p}(\R;\WSR{\sigma_x}{\td p}(\R^d))$ for all $\sigma_t\in[0,\td\kappa_t)$ and $\sigma_x\in[0,\td\kappa_x)$. Furthermore, there is an $r\in (\tilde p, m+1-\gamma)$, such that
 \begin{align}\label{cor:av_main_est}
  \norm{\overline{f}}_{\WSR{\sigma_t}{\td p}(\WSR{\sigma_x}{\td p})}\lesssim \||v|^{1-\gamma}g_{0}\|_{\calm_{TV}}+\||v|^{-\gamma}g_{1}\|_{\calm_{TV}} + \|f\|_{\LR{1}_{t,x,v}\cap \LR{\infty}_{t,x,v}} +  \|\overline{f}\|_{\LR{r}_{t,x}}.
 \end{align}
\end{cor}
\begin{proof} As we need to pass from homogeneous spaces (the output of Lemma \ref{lem:av} and Lemma \ref{lem:av2}) to a non-homogeneous space, our strategy is to invoke Lemma \ref{lem:emb_2} and Lemma \ref{lem:emb_1}. The input to Lemma \ref{lem:emb_2} requires four pieces of information, namely control of $\overline{f}$ in $\LR{\td p}(\R^{d+1})$, $\td{L}^{\td p}_x\dot B^{\sigma_t}_{\td p,\infty}$, $\td{L}^{\td p}_t\dot B^{\sigma_x}_{\td p,\infty}$ and $S^{\overline{\sigma}}_{\td p,\infty}\dot B$. Since the control of $\overline{f}$ in $\LR{\td p}(\R^{d+1})$ is ensured by assumption, we concentrate on the other three contributions. Note that the main difficulty lies in the condition that both the integrability exponent and the orders of differentiability have to match exactly.

\smallskip\noindent
 \textit{Step 1:} $\overline{f}\in S^{\overline{\sigma}}_{\td p,\infty}\dot B$.  
 Let $r\in(\td p,m+1-\gamma)$ to be chosen in Step 3. We claim that there exist functions $k_t,k_x:(0,\infty)\to(0,\infty)$ with $k_t(\eps),k_x(\eps)\to 0$ as $\eps\to 0$, such that it holds for all $\eps\ll 1$
 \begin{align}\label{cor:av_est}
  \norm{\overline{f}}_{S^{\overline{\sigma}}_{\td p,\infty}\dot B}\lesssim \||v|^{1-\gamma}g_{0}\|_{\calm_{TV}}+\||v|^{-\gamma}g_{1}\|_{\calm_{TV}} + \|f\|_{\LR{1}_{t,x,v}\cap \LR{\infty}_{t,x,v}} + \|\overline{f}\|_{\LR{r}_{t,x}},
 \end{align}
 where we have used the notation $\sigma_t:=\td\kappa_t-k_t(\eps)$ and $\sigma_x:=\td\kappa_x-k_x(\eps)$.
 
 We apply Lemma \ref{lem:av} with $\mu=1$, $\rho=1-\eps$, and $s:=s_\eps\in (0,1)$, where $s_\eps$ is chosen so that the integrability assertion in \eqref{lem:av_constants} reads $p=\td p$; this is possible for $\rho$ close to $1$ in view of Remark \ref{rem:av}.
 Moreover, we may choose $\vt\in(0,1)$ such that for $\kappa_t$ and $\kappa_x$ defined through \eqref{lem:av_constants} satisfy $\vt\kappa_t=\td\kappa_t-k_t(\eps)$ and $\vt\kappa_x=\td\kappa_x-k_x(\eps)$ for some functions $k_t$ and $k_x$ as above. Then for $1<q_0<\td p<q_1<m+1-\gamma$ so that
 \begin{align*}
  \frac{1}{q_0}=1-\vt+\frac{\vt}{\td p}, \qquad
  \frac{1}{q_1}=\frac{1-\vt}{r}+\frac{\vt}{\td p},
 \end{align*}
 in view of \eqref{lem:av_est2} we obtain that
 \begin{align*}
  \|\overline{f}_{l,j}\|_{\LR{q_i}_{t,x}}\lesssim 2^{-l\vt\kappa_t}2^{-j\vt\kappa_x}(\||v|^{1-\gamma}g_{0}\|_{\calm_{TV}}+\||v|^{-\gamma}g_{1}\|_{\calm_{TV}} + \|f\|_{\LR{\rhon}_{t,x,v}} + \|\overline{f}\|_{\LR{1}_{t,x}\cap \LR{r}_{t,x}}),
 \end{align*}
 for $i=0,1$, where we recall the notation $\overline{f}_{l,j}:=\int \calf^{-1}_{t,x}[\eta_l\vp_j \calf_{t,x}f] \dd v$.
 Since $(\LR{q_0}_{t,x},\LR{q_1}_{t,x})_{\t,\td p}=\LR{\td p}_{t,x}$ for an appropriate $\t\in(0,1)$, we thus obtain
 \begin{align*}
  \|\overline{f}_{l,j}\|_{\LR{\td p}_{t,x}}\lesssim 2^{-l\vt\kappa_t}2^{-j\vt\kappa_x}&\left(\||v|^{1-\gamma}g_{0}\|_{\calm_{TV}}+\||v|^{-\gamma}g_{1}\|_{\calm_{TV}} \right.\\
  & \left. \ + \|f\|_{\LR{\rhon}_{t,x,v}} + \|\overline{f}\|_{\LR{1}_{t,x}} + \|\overline{f}\|_{\LR{r}_{t,x}}\right),
 \end{align*}
 which after multiplying by $2^{l\vt\kappa_t}2^{j\vt\kappa_x}$ and taking the supremum over $l,j\in \Z$ yields
 \begin{align}\label{cor:av_est_b}
  \norm{\overline{f}}_{S^{\vt(\kappa_t,\kappa_x)}_{\td p,\infty}\dot B}\lesssim \||v|^{1-\gamma}g_{0}\|_{\calm_{TV}}+\||v|^{-\gamma}g_{1}\|_{\calm_{TV}} + \|f\|_{\LR{\rhon}_{t,x,v}} + \|\overline{f}\|_{\LR{1}_{t,x}} + \|\overline{f}\|_{\LR{r}_{t,x}}.
 \end{align}
 By the estimate $\|f\|_{\LR{\rhon}_{t,x,v}} + \|\overline{f}\|_{\LR{1}_{t,x}}\lesssim \|f\|_{\LR{1}_{t,x,v}} + \|f\|_{\LR{\infty}_{t,x,v}}$, this gives \eqref{cor:av_est}.
 
 \smallskip\noindent
 \textit{Step 2:} $\overline{f}\in \td{L}^{\td p}_t\dot B^{\sigma_x}_{\td p,\infty}$.  
 In this step we establish
 \begin{align}\label{cor:av_est2}
  \norm{\overline{f}}_{\td{L}^{\td p}\dot B^{\sigma_x}_{\td p,\infty}}\lesssim \||v|^{1-\gamma}g_{0}\|_{\calm_{TV}}+\||v|^{-\gamma}g_{1}\|_{\calm_{TV}} + \|f\|_{\LR{1}_{t,x,v}\cap \LR{\infty}_{t,x,v}}  + \|\overline{f}\|_{\LR{r}_{t,x}}.
 \end{align}
 Choose
 \begin{align*}
  \rho:=\frac{(\td p-1)(m-\gamma)}{1+\td p(m-1-\gamma)}. 
 \end{align*}
 We claim that $\rho$ is positive and well-defined: Since the nominator is positive due to $\td p>1$ and $m>\gamma$, it remains to check that the denominator is positive. This is obvious for $\gamma\le m-1$. For $\gamma> m-1$, we observe that
 \begin{align*}
  \td p<m+1-\gamma<\frac{1}{1+\gamma-m},
 \end{align*}
 which implies $1+\td p(m-1-\gamma)>0$. Moreover, $\td p<m+1-\gamma$ can be rewritten as $(\td p-1)(m-\gamma)<1+\td p(m-1-\gamma)$, so that $\rho\in(0,1)$. Hence, we may apply Lemma \ref{lem:av} with this choice of $\rho$ and with $s=1$. One checks that in this case the integrability and differentiability exponents in \eqref{lem:av_constants} read $p=\td p$, $\kappa_t=0$, and $\kappa_x=\frac{p-1}{p}\frac{2}{m-\gamma}$. We observe that $\kappa_x\ge \td\kappa_x$ and hence we find $\vt\in (0,1)$ such that $\vt\kappa_x=\td\kappa_x-k_x(\eps)$.
 The same interpolation argument as in Step 1 gives now the estimate \eqref{cor:av_est2}.
 
 \smallskip\noindent
 \textit{Step 3:} $\overline{f}\in \td{L}^{\td p}_x\dot B^{\sigma_t}_{\td p,\infty}$. 
 In this step we show that there is some $r\in (\td p, m+1-\gamma)$ such that
 \begin{align}\label{cor:av_est3}
  &\norm{\overline{f}}_{\td{L}^{\td p}_x\dot B^{\sigma_t}_{\td p,\infty}}\lesssim \||v|^{1-\gamma}g_{0}\|_{\calm_{TV}}+\||v|^{-\gamma}g_{1}\|_{\calm_{TV}} + \|f\|_{\LR{\infty}_{t,x,v}} + \|\overline{f}\|_{\LR{r}_{t,x}}.
 \end{align}
 We apply Lemma \ref{lem:av2} with $\mu=1$ and $\rho=1$. In this case, \eqref{lem:av2_constants} reads $p=2-\gamma$ and $\kappa_t=\frac{1}{2-\gamma}$. Since $\td p>2-\gamma$, we have $\td\kappa_t<\kappa_t$. Hence, we can choose $\vt\in (0,1)$, such that $\vt\kappa_t=\td\kappa_t-k_t(\eps)$. In particular,
 \begin{align*}
  \vt<\frac{\td\kappa_t}{\kappa_t}=\frac{m+1-\gamma-\td p}{\td p}\frac{2-\gamma}{m-1}<\frac{2-\gamma}{\td p}, 
 \end{align*}
 so that $r=\frac{\td p(2-\gamma)(1-\vt)}{2-\gamma-\vt \td p}$ is well defined. Since $r$ is increasing in $\vt$ due to $\td p>2-\gamma$, we see that $r\in (\td p,m+1-\gamma)$. We have $\frac{1}{\td p}=\frac{1-\vartheta}{r}+\frac{\vartheta}{p}$, and hence Lemma \eqref{lem:av2} gives estimate \eqref{cor:av_est3}.
 
 \smallskip\noindent
 \textit{Step 4:} Conclusion. 
 Since $\overline{f}\in \LR{\td p}_{t,x}$ by assumption, Lemma \ref{lem:emb_2} combined with Lemma \ref{lem:emb_1} yields the result.
\end{proof}


\begin{cor}\label{cor:av3}
 Let $m\in\vpp{1,\infty}$, $\gamma\in\vpp{-\infty,m}$, and let $f\in \LR{1}_{t,x,v}\cap \LR{\infty}_{t,x,v}$ be a solution to \eqref{av_eqn}. Let $g_0$ and $g_1$ be as in Lemma \ref{lem:av} and assume additionally
 \begin{align*}
  |g_{0}|(t,x,v) \in \calm_{TV}(\R_{t}\times\R_{x}^{d}\times\R_{v}).
 \end{align*}
 Assume $\td p\in (2-\gamma,m+1-\gamma)\cap (1,m+1-\gamma)$ and define
 \begin{align*}
  \td\kappa_t:= \frac{m+1-\gamma-\td p}{\td p}\frac{1}{m-1}, \qquad
  \td\kappa_x:= \frac{\td p-2+\gamma}{\td p}\frac{2}{m-1}.
 \end{align*}
 If $\overline{f}\in \LR{r}(\R^{d+1})\cap \LR{1}(\R;\LR{\td p}(\R^d))$ for all $r\in [1,m+1-\gamma)$, where $\overline{f}(t,x):=\int f(t,x,v)\dd v$, and if $\int|\vf|^{m-1} \f \dd v \in \LR{1}(\R^{d+1})$, then $\overline{f}\in \WSR{\sigma_t}{\td p}(\R;\WSR{\sigma_x}{\td p}(\R^d))$ for all $\sigma_t\in[0,\td\kappa_t)$ and $\sigma_x\in[0,\td\kappa_x)$. Furthermore, there is an $r\in (\tilde p, m+1-\gamma)$, such that
 \begin{align}\label{cor:av3_main_est}
 \begin{split}
  & \norm{\overline{f}}_{\WSR{\sigma_t}{\td p}(\WSR{\sigma_x}{\td p})}\lesssim \|g_{0}\|_{\calm_{TV}} + \||v|^{1-\gamma}g_{0}\|_{\calm_{TV}}+\||v|^{-\gamma}g_{1}\|_{\calm_{TV}}\\
  & \qquad  \qquad + \|f\|_{\LR{1}_{t,x,v}\cap \LR{\infty}_{t,x,v}} 
  + \|\overline{f}\|_{\LR{1}_t\LR{\td p}_x\cap \LR{r}_{t,x}} + \|\int|\vf|^{m-1} \f \dd v\|_{\LR{1}_{t,x}}.
 \end{split}
 \end{align}
\end{cor}
\begin{proof} It suffices to adapt Step 3 of the proof of Corollary \ref{cor:av}, that is the control of $\overline{f}$ in $\td{L}^{\td p}_x\dot B^{\sigma_t}_{\td p,\infty}$.
 
 \smallskip\noindent
 \textit{Step 3}. $\overline{f}\in \td{L}^{\td p}_x\dot B^{\sigma_t}_{\td p,\infty}$. 
  In this step we show that there is some $r\in (\td p, m+1-\gamma)$ such that
 \begin{align}\label{cor:av3_est3}
 \begin{split}
  \norm{\overline{f}}_{\td{L}^{\td p}_x\dot B^{\sigma_t}_{\td p,\infty}}& \lesssim \|g_{0}\|_{\calm_{TV}} + \||v|^{1-\gamma}g_{0}\|_{\calm_{TV}}+\||v|^{-\gamma}g_{1}\|_{\calm_{TV}} \\ 
  & + \|f\|_{\LR{1}_{t,x,v}\cap \LR{\infty}_{t,x,v}}  + \|\overline{f}\|_{\LR{1}_t\LR{\td p}_x\cap \LR{r}_{t,x}} + \|\int|\vf|^{m-1} \f \dd v\|_{\LR{1}_{t,x}}.
 \end{split}
 \end{align}
 We split $f$ into three contributions
 \begin{align*}
  f&=\calf^{-1}_t\psi_0(\tau)\calf_t f + \calf^{-1}_{t,x}(1-\psi_0(\tau))(1-\phi_0(\xi)) \calf_{t,x} f \\
   & \ \ \ \ + \calf^{-1}_{t,x}(1-\psi_0(\tau))\phi_0(\xi) \calf_{t,x} f \\
  &=:f^1+f^2+f^3.
 \end{align*}
 The low time-frequency part $f^1$ can be estimated in view of Lemma \ref{lem:emb_0} and Bernstein's Lemma (cf.\@ \cite[Lemma 2.1]{BCD11}) via
 \begin{align}\label{cor:av3_est5}
  \|\overline{f}^1\|_{\td{L}^{\td p}_x\dot B^{\sigma_t}_{\td p,\infty}}\lesssim \|\overline{f}^1\|_{\td{L}^{\td p}_x B^{\sigma_t}_{\td p,\infty}}\lesssim \|\overline{f}^1\|_{\WSR{\sigma_t+\eps}{\td p}(\LR{\td p}_x)}\lesssim \|\overline{f}\|_{\LR{1}_{t}\LR{\td p}_x}.
 \end{align}
 Next, we apply Lemma \ref{lem:av} with $\mu=1$, sufficiently large $\rho\in(0,1)$ and sufficiently small $s\in (\frac{\gamma-1}{m-1},1]$ so that \eqref{lem:av_constants} implies $p<\td p$ and $\kappa_t>\td\kappa_t$. Hence, we can choose $\vt\in (0,1)$, such that $\td\kappa_t>\vt\kappa_t>\td\kappa_t-k_t(\eps)$. In particular, in light of Remark \ref{rem:av}
 \begin{align*}
  \vt<\frac{\td\kappa_t}{\kappa_t}=\frac{m+1-\gamma-\td p}{m+\rho-\gamma-p\rho+p(1-\rho)(\gamma-m)}\frac{p\rho}{\td p}<\frac{p}{\td p}, \qquad \tif \ 1-\rho\ll 1,
 \end{align*}
 so that $r=\frac{\td p p (1-\vt)}{p-\vt \td p}$ is well defined. Since $r$ is increasing in $\vt$ due to $\td p>p$, we see that $r\in (\td p,m+1-\gamma)$. We have $\frac{1}{\td p}=\frac{1-\vartheta}{r}+\frac{\vartheta}{p}$, and hence Lemma \ref{lem:av} gives
 \begin{align*}
  &\norm{\overline{f}}_{S^{\vt\overline{\kappa}}_{\td p,\infty}\dot B}\lesssim \||v|^{1-\gamma}g_{0}\|_{\calm_{TV}}+\||v|^{-\gamma}g_{1}\|_{\calm_{TV}} + \|f\|_{\LR{1}_{t,x,v}\cap \LR{\infty}_{t,x,v}} + \|\overline{f}\|_{\LR{r}_{t,x}}.
 \end{align*}
 Thus, since $f^2$ is supported only on $\eta_l\vp_j$ for non-negative $l,j\in\Z$, Lemma \ref{lem:emb_0} and Lemma \ref{lem:emb_1} show in view of the definition of the homogeneous and non-homogeneous Besov spaces and $\sigma_t<\vt\kappa_t$ as well as $0<\vt\kappa_x$
 \begin{align*}
  \norm{\overline{f}^2}_{\td{L}^{\td p}_x\dot B^{\sigma_t}_{\td p,\infty}}=\norm{\overline{f}^2}_{\td{L}^{\td p}_x B^{\sigma_t}_{\td p,\infty}}
  \lesssim \norm{\overline{f}^2}_{S^{\vt\overline{\kappa}}_{\td p,\infty} B}  = \norm{\overline{f}}_{S^{\vt\overline{\kappa}}_{\td p,\infty}\dot B}. 
 \end{align*}
 Thus,
 \begin{equation}\label{cor:av3_est4}
 \norm{\overline{f}^2}_{\td{L}^{\td p}_x\dot B^{\sigma_t}_{\td p,\infty}}\lesssim \||v|^{1-\gamma}g_{0}\|_{\calm_{TV}}\!+\!\||v|^{-\gamma}g_{1}\|_{\calm_{TV}}\!+\! \|f\|_{\LR{1}_{t,x,v}\cap \LR{\infty}_{t,x,v}}\!+\!\|\overline{f}\|_{\LR{r}_{t,x}}.
 \end{equation}
 It remains to estimate the contribution of $f^3$. For $l\in \Z$, we introduce $f^3_l:=\calf^{-1}_t\eta_l(\tau)\calf_t f^3$. Since $f^3_l=0$ for $l<0$, we may concentrate on the case $l\ge 0$. Observe that $f^3_{l}$ solves the equation
  \begin{align*}
   f^3_{l}= & -m|\vf|^{m-1}\calf^{-1}_{t,x}\frac{|\xi|^2}{i\tau}\eta_l(\tau)\phi_0(\xi)\calf_{t,x}\f + \calf^{-1}_{t,x}\frac{\phi_0(\xi)}{i\tau}\calf_{t,x}  g_{0,l}\\ 
   & + \calf^{-1}_{t,x}\frac{\phi_0(\xi)}{i\tau}\calf_{t,x} \partial_\vf g_{1,l}.
  \end{align*}
  Integrating in $v$, we obtain
  \begin{align*}
   \overline{f}^3_{l}=-m \int |\vf|^{m-1}\calf^{-1}_{t,x}\frac{|\xi|^2}{i\tau}\eta_l(\tau)\phi_0(\xi)\calf_{t,x}\f \dd v + \calf^{-1}_{t,x}\frac{1}{i\tau}\phi_0(\xi)\calf_{t,x} \int g_{0,l,j} \dd v.
  \end{align*}
  Since $|\xi|^2$ acts as a constant multiplier on the support of $\phi_0$ and $\tau^{-1}$ acts as a constant multiplier of order $2^{-l}$ on the support of $\eta_l$, it follows by Bernstein's Lemma
  \begin{align*}
   \norm{\overline{f}^3_{l}}_{\LR{\td p}_{t,x}}&\lesssim 2^{l(1-\frac{1}{\td p})}\norm{\overline{f}^3_{l}}_{\LR{1}_{t,x}}\lesssim 2^{-l\frac{1}{\td p}}(\|\int|\vf|^{m-1} \f \dd v\|_{\LR{1}_{t,x}} + \|g_0\|_{\calm_{TV}}).
  \end{align*}
  Since $\td p>2-\gamma$, we have
  \begin{align*}
   \sigma_t<\td\kappa_t=\frac{m+1-\gamma-\td p}{\td p} \frac{1}{m-1}<\frac{1}{\td p}.
  \end{align*}
  In view of $l\ge 0$ this yields
  \begin{align*}
   \norm{\overline{f}^3_{l}}_{\LR{\td p}_{t,x}}&\lesssim 2^{-l\sigma_t}(\|\int|\vf|^{m-1} \f \dd v\|_{\LR{1}_{t,x}} + \|g_0\|_{\calm_{TV}}).
  \end{align*}
  Multiplying by $2^{l\sigma_t}$ and taking the supremum over $l\in\Z$, we conclude
  \begin{align}\label{cor:av3_est6}
   \norm{\overline{f}^3}_{\td{L}^{\td p}_x\dot B^{\sigma_t}_{\td p,\infty}}&\lesssim \|\int|\vf|^{m-1} \f \dd v\|_{\LR{1}_{t,x}} + \|g_0\|_{\calm_{TV}}.
  \end{align}
 Collecting \eqref{cor:av3_est5}, \eqref{cor:av3_est4} and \eqref{cor:av3_est6}, we arrive at \eqref{cor:av3_est3}.
\end{proof}

%
\section{Application to Porous Medium Equations}\label{App_PME}

In this section, we provide proofs of our main results by applying the averaging lemmata obtained in the previous section to entropy solutions to \eqref{pme_sys}.

 \begin{proof}[Proof of Theorem \ref{lem:pme}]
We first argue that we have $u\in \LR{s}_{t,x}$ for all $s\in[1,m-1+\rho)$. Since $T<\infty$,  Theorem \ref{thm:wp-kinetic} gives
\begin{align}\label{lem:pme_L1bound}
 \norm{u}_{\LR{1}_{t,x}}\lesssim \sup_{t\in[0,T]}\|\uf(t)\|_{\LR{1}_{x}} \lesssim \norm{u_0}_{\LR{1}_x}+\norm{S}_{\LR{1}_{t,x}},
\end{align}
so that we may concentrate on $s>1$.
Let $\f$ be the kinetic function corresponding to $\uf$ and solving \eqref{pme_sys_kin}. 
 In order to apply Corollary \ref{cor:av0} with $\mu=1$ and $\sigma_x=0$, we need to extend \eqref{pme_sys_kin} to all times $t\in \R$, which can be achieved by multiplication with a smooth cut-off function $\vp\in \CRci(0,T)$ with $0\le \vp\le 1$. Hence, we set $g_0:=\delta_{v=u(t,x)}S+\partial_t\vp f$ and $g_1:=q$. Let $\gamma:=2-\rho$, so that $s\in (1,m+1-\gamma)$.
 From \eqref{cor:av0_main_est} we obtain
 \begin{align*}
  \norm{\vp\uf}_{L^{s}_{t,x}}&\lesssim \||v|^{\rhoe-1}g_{0}\|_{\calm_{TV}}+\||v|^{\rhoe-2}g_{1}\|_{\calm_{TV}} + \|\vp f\|_{\LR{1}_{t,x,v}\cap \LR{\infty}_{t,x,v}} + \|\vp \uf\|_{\LR{1}_{t,x}\cap \LR{s}_{t}\LR{1}_{x}} \\
  &\lesssim \||v|^{\rhoe-1}g_{0}\|_{\calm_{TV}}+\||v|^{\rhoe-2}g_{1}\|_{\calm_{TV}} + \|f\|_{\LR{1}_{t,x,v}\cap \LR{\infty}_{t,x,v}} + \sup_{t\in[0,T]}\|\uf(t)\|_{\LR{1}_{x}}.
 \end{align*}
 We note that since trivially $f\in \LR{\infty}_{t,x,v}$ with norm bounded by $1$, estimate \eqref{lem:pme_L1bound} gives
 \begin{align*}
  \lefteqn{\norm{f}_{\LR{1}_{t,x,v} \cap \LR{\infty}_{t,x,v}} + \sup_{t\in[0,T]}\|\uf(t)\|_{\LR{1}_{x}}} \\
  & & \lesssim  \norm{u}_{\LR{1}_{t,x}}+ 1 + \sup_{t\in[0,T]}\|\uf(t)\|_{\LR{1}_{x}}\lesssim \norm{u_0}_{\LR{1}_x}+\norm{S}_{\LR{1}_{t,x}}+ 1.
 \end{align*}
Next, we check that $|v|^{\rhoe-1}g_0 \in \calm_{TV}$. Indeed, we observe that $(\rhoe-1)\rhoe':=\rhoe$, and hence, applying Lemma \ref{lem:ph-est},
\begin{align*}
 \||v|^{\rhoe-1}g_0\|_{\calm_{TV}}&=\||v|^{\rhoe-1}(\delta_{v=u(t,x)}S+\partial_t\vp f)\|_{\calm_{TV}}\lesssim \||u|^{\rhoe-1}S\|_{L^{1}_{t,x}}+ \|\partial_t\vp |u|^{\rhoe}\|_{L^1_{t,x}} \\
 &\lesssim \||u|^{(\rhoe-1)\rhoe'}\|_{L^{1}_{t,x}} +\||S|^{\rhoe}\|_{L^{1}_{t,x}} + \|\partial_t\vp |u|^{\rhoe}\|_{L^1_{t,x}} \\ 
 &\lesssim \|u_0\|_{L^{\rhoe}_{x}}^{\rhoe} + \|S\|_{L^{\rhoe}_{t,x}}^{\rhoe} + \|\partial_t\vp |u|^{\rhoe}\|_{L^1_{t,x}}.
\end{align*}
Utilizing Lemma \ref{lem:ph-est} once more to the effect of
\begin{align*}
 \||v|^{\rhoe-2}g_1\|_{\calm_{TV}}=\||v|^{\rhoe-2}q\|_{\calm_{TV}}\lesssim \|u_0\|_{L^{\rhoe}_{x}}^{\rhoe} + \|S\|_{L^{\rhoe}_{t,x}}^{\rhoe},
\end{align*}
we obtain 
\begin{align*}
  \|\vp u\|_{\LR{s}_{t,x}}\lesssim \|u_0\|_{L^1_x\cap L^{\rhoe}_{x}}^{\rhoe} + \|S\|_{L^1_{t,x}\cap L^{\rhoe}_{t,x}}^{\rhoe}+\|\partial_t\vp |u|^{\rhoe}\|_{L^1_{t,x}}+1.
 \end{align*}
We may set $\vp_n(t)=\psi(nt)-\psi(nt-T/2)$, where $\psi\in\CRi(\R)$ with $0\le \psi\le 1$, $\supp\psi\subset(0,\infty)$, $\psi(t)=1$ for $t>T/2$ and $\|\partial_t\psi\|_{L^1}=1$. For $n\to\infty$, $\vp_n$ converges to $1_{[0,T]}$ in the supremum norm, while $\partial_t\vp_n$ is a smooth approximation of $\delta_{\set{t=0}}-\delta_{\set{t=T}}$. Therefore, $\|\vp_n u\|_{\LR{s}_{t,x}}\to \|u\|_{\LR{s}_{t,x}}$ and by an application of Lemma \ref{lem:ph-est} $\|\partial_t\vp_n |u|^{\rhoe}\|_{L^1_{t,x}}\to \||u|(0)^\rhoe-|u|(T)^\rhoe\|_{L^1_{x}}\lesssim \|u_0\|_{L^\rhoe_{x}}^\rhoe+\|S\|_{L^{\rhoe}_{t,x}}^{\rhoe}$, so that $u\in L^{s}([0,T]\times \R^d)$ and
\begin{align}\label{lem:pme_Lmbound}
  \|u\|_{\LR{s}_{t,x}}\lesssim \|u_0\|_{L^1_x\cap L^{\rhoe}_{x}}^{\rhoe} + \|S\|_{L^1_{t,x}\cap L^{\rhoe}_{t,x}}^{\rhoe}+1.
 \end{align}

\newcounter{counter:lem_pme} 
 \refstepcounter{counter:lem_pme} 
  
  (\arabic{counter:lem_pme}). 
  We apply Corollary \ref{cor:av0} once more. Let $f$, $\vp$, $g_0$, $g_1$ and $\gamma$ be as above. Then, in particular $p\in (1,\frac{m+1-\gamma}{\mu})$.
 From \eqref{cor:av0_main_est} we obtain
 \begin{align*}
  \lefteqn{\norm{\vp\uf^{[\mu]}}_{L^{p}_{t} (\WSR{\sigma_x}{p})}}\\
  &&\lesssim \||v|^{1-\gamma}g_{0}\|_{\calm_{TV}}+\||v|^{-\gamma}g_{1}\|_{\calm_{TV}} + \|f\|_{\LR{1}_{t,x,v}\cap \LR{\infty}_{t,x,v}} + \|\uf^{[\mu]}\|_{\LR{1}_{t,x}\cap \LR{p}_t\LR{1}_x}.
 \end{align*}
 The first three contributions on the right hand side are estimated as above. For the last contribution, we note $1\le \mu<p\mu$ and thus
 \begin{align*}
  \|\uf^{[\mu]}\|_{\LR{1}_{t,x}\cap \LR{p}_t\LR{1}_x}&\lesssim \|\uf^{[\mu]}\|_{\LR{p}_t\LR{1}_x}=\|\uf\|_{\LR{p\mu}_{t}\LR{\mu}_x}^{\mu}\lesssim (\|\uf\|_{\LR{p\mu}_{t}\LR{1}_x} + \|\uf\|_{\LR{p\mu}_{t,x}})^{\mu} \\
  &\lesssim (\sup_{t\in[0,T]}\|\uf(t)\|_{\LR{1}_x} + \|\uf\|_{\LR{p\mu}_{t,x}})^{\mu} \lesssim \sup_{t\in[0,T]}\|\uf(t)\|_{\LR{1}_x}^\mu + \|\uf\|_{\LR{p\mu}_{t,x}}^\mu +1.
 \end{align*}
 Furthermore, \eqref{lem:pme_L1bound} together with \eqref{lem:pme_Lmbound} applied with $s=p\mu\in(1,m-1+\rhoe)$ shows
 \begin{align*}
\sup_{t\in[0,T]}\|\uf(t)\|_{\LR{1}_x}^\mu + \|\uf\|_{\LR{p\mu}_{t,x}}^\mu +1 \lesssim \norm{\uf_0}_{L_x^1\cap L_x^{\rhoe}}^{\mu\rhoe} + \norm{S}_{L_{t,x}^1\cap L_{t,x}^{\rhoe}}^{\mu\rhoe} + 1.
\end{align*}
Hence, arguing as above by taking the limit $\vp_n\to 1_{[0,T]}$, we obtain $u^{[\mu]}\in L^{p}(\R; \WSR{\sigma_x}{p}(\R^d))$ and \eqref{lem:pme_est2}.

\refstepcounter{counter:lem_pme} 
  
  (\arabic{counter:lem_pme}).\label{lem:pme_2} The proof is similar to the first part, but we use Corollary \ref{cor:av} instead of Corollary \ref{cor:av0}. Again we localize in time by multiplying with a smooth cut-off function $\vp\in \CRci(0,T)$ with $0\le\vp\le 1$ and set $g_0$ and $g_1$ as before. Choose $\gamma:=2-\rhoe$, so that $p\in(2-\gamma,m+1-\gamma)$. 
From \eqref{cor:av_main_est} in Corollary \ref{cor:av} we obtain
\begin{align*}
 \norm{\vp u}_{\WSR{\sigma_t}{p}(\WSR{\sigma_x}{p})}&\lesssim \||v|^{1-\gamma}g_{0}\|_{\calm_{TV}}+\||v|^{-\gamma}g_{1}\|_{\calm_{TV}} + \|f\|_{\LR{1}_{t,x,v} \cap \LR{\infty}_{t,x,v}} + \|u\|_{\LR{r}_{t,x}},
\end{align*}
where $r\in (\rhoe,m-1+\rhoe)$.
The terms involving $g_0$, $g_1$ and $f$ can be estimated as above, while the $\LR{r}_{t,x}$-norm of $u$ can be estimated by \eqref{lem:pme_Lmbound}. 
Choosing $\vp_n$ as above, we hence infer that $\vp_n u$ is bounded in $\WSR{\sigma_t}{p}(0,T;\WSR{\sigma_x}{p}(\R^d))$ and
\begin{align*}
 \sup_{n\in\N}\norm{\vp_n u}_{\WSR{\sigma_t}{p}(\WSR{\sigma_x}{p})}&\lesssim \|u_0\|_{L^1_x\cap L^{\rhoe}_{x}}^{\rhoe} + \|S\|_{L^1_{t,x}\cap L^{\rhoe}_{t,x}}^{\rhoe}+1.
\end{align*}
Since $\vp_n u\to u 1_{[0,T]}$ in the sense of distributions, we obtain the result by the weak lower semi-continuity of the norm in $\WSR{\sigma_t}{p}(0,T;\WSR{\sigma_x}{p}(\R^d))$.  \qedhere
\end{proof}

 \begin{proof}[Proof of Corollary \ref{cor:pme}]\mbox{ }
   
   \newcounter{counter:cor_pme} 
 \refstepcounter{counter:cor_pme} 
  
  \medskip
  (\arabic{counter:cor_pme}).
   Let $\sigma_x\in[0,\frac{2\mu}{m})$. We apply Theorem \ref{lem:pme}~(i) with $p=\frac{m}{\mu}$ for sufficiently small $\eta\in(1,\rhoe]$ so that $\sigma_x<\frac{\mu p -1}{p}\frac{2}{m-2+\eta}=\frac{2\mu}{m}\frac{m-1}{m-2+\eta}$ and observe that for all $q\in[1,p]$ we have the embedding $L^{p}(0,T; \WSR{\sigma_x}{p}(\R^d))\subset L^{q}(0,T; \WSR{\sigma_x}{q}(\calo))$.
   
   \medskip
   \refstepcounter{counter:cor_pme} 
    (\arabic{counter:cor_pme})\label{cor:pme_2}. For $s>0$ we have, with $p=s(m-1)+1\in(1,m]$,
   \begin{align*}
    \kappa_t=\frac{1-s}{s(m-1)+1}=\frac{m-p}{p}\frac{1}{m-1}, \quad
    \kappa_x=\frac{2s}{s(m-1)+1}=\frac{p-1}{p}\frac{2}{m-1}.
   \end{align*}
   Hence, in this case the assertion follows by an application of Theorem \ref{lem:pme}~(ii) with sufficiently small $\eta\in(1,\rhoe]$ such that $p>\rhoe$ and $\sigma_x<\frac{p-\rhoe}{p}\frac{2}{m-1}$ combined with the embedding \[\WSR{\sigma_t}{p}(0,T;\WSR{\sigma_x}{p}(\R^d))\subset \WSR{\sigma_t}{q}(0,T;\WSR{\sigma_x}{q}(\calo)).\]
   
   If $s=0$ and $\sigma_t\in[0,1)$, we may choose $s_0>0$ such that $\sigma_t<\frac{1-s_0}{s_0(m-1)+1}=:\kappa_t(s_0)$, and  the result follows by the embedding 
   \[
   \WSR{\kappa_t(s_0)}{s_0(m-1)+1}(0,T;\LR{s_0(m-1)+1}(\calo))\subset \WSR{\sigma_t}{1}(0,T;\LR{1}(\calo)).  \qedhere
   \]
 \end{proof}

\begin{proof}[Proof of Theorem \ref{cor:pme_l1}]
 The proof is similar to the one of Theorem \ref{lem:pme} (\ref{lem:pme_2}), but we discriminate between small and large velocity contributions to the kinetic function. Let $\f$ be the kinetic function corresponding to $\uf$ and solving \eqref{pme_sys_kin}. We extend again to all times $t\in \R$ by multiplying with a smooth cut-off function $\vp\in \CRci(0,T)$ with $0\le \vp\le 1$. Further, we split $f=:f^<+f^>$ and $q=:q^<+q^>$ into a small-velocity and a large-velocity part by multiplying with a smooth cut-off function $\psi_0$ respectively $\psi_1:=1-\psi_0$ in $v$. This gives rise to the two equations
  \begin{align*}
   \partial_t(\vp\f^<)-m|\vf|^{m-1}\Delta_x(\vp\f^<) &= \vp\psi_0\delta_{v=u(t,x)}S + \partial_\vf (\vp q^<) - \vp q\partial_\vf\psi_0 + \partial_t\vp f^<, \\
   \partial_t(\vp \f^>)-m|\vf|^{m-1}\Delta_x(\vp \f^>) &= \vp\psi_1\delta_{v=u(t,x)}S + \partial_\vf (\vp q^>) + \vp q\partial_\vf\psi_0 + \partial_t\vp f^>,
  \end{align*}
  Integrating $f^<$ and $f^>$ in $v$, we obtain a decomposition of $u=u^<+u^>$.
 
  The proof proceeds in several steps: In first the three steps, we argue that $u\in \LR{s}(0,T;\LR{s}(\R^d))$ for all $s\in[1,m+\frac{2}{d})$ if $d\ge 2$ and $s\in[1,m+1)$ if $d=1$. With this additional bound, we can conclude the higher-order estimates in the last three steps of the proof. We only detail the proof for $d\ge 2$, the case $d=1$ being similar.
  
  \newcounter{cor:pme_l1_prf} 
  \refstepcounter{cor:pme_l1_prf} 
  \textit{Step} \arabic{cor:pme_l1_prf}\label{cor:pme_l1_prf_st1}. In this step we establish for $\rho\in (m,\frac{md}{d-2})$ the bound
  \begin{align}\label{cor:pme_l1_s1}
  \|u^<\|_{\LR{m}_{t}\LR{\rho}_x}\lesssim \|u_0\|_{L^1_x} + \|S\|_{L^1_{t,x}} +1.
  \end{align}
  Set $g_0:=\vp\psi_0\delta_{v=u(t,x)}S + \partial_t\vp f^< - \vp q\partial_\vf\psi_0$, $g_1:=\vp q^<$, and
  \begin{align*}
   \sigma_x:=\frac{d}{m}-\frac{d}{\rho}\in\vpp{0,\frac{2}{m}}.
  \end{align*}
 Consequently, we may choose $\gamma\in(0,1)$ so large that $\sigma_x\in[0,\frac{m-1}{m}\frac{2}{m-\gamma})$.
 From Corollary \ref{cor:av0} applied with $\mu=1$ and $q=m$ we obtain
 \begin{align*}
  \lefteqn{\norm{\vp\uf^<}_{L^{m}_{t}\WSR{\sigma_x}{m}_x}}\\
  &&\lesssim \||v|^{1-\gamma}g_{0}\|_{\calm_{TV}}+\||v|^{-\gamma}g_{1}\|_{\calm_{TV}} + \|\vp f^<\|_{\LR{1}_{t,x,v}\cap \LR{\infty}_{t,x,v}} + \|\vp \uf^<\|_{\LR{1}_{t,x}\cap \LR{m}_{t}\LR{1}_{x}} \\
  &&\lesssim \||v|^{1-\gamma}g_{0}\|_{\calm_{TV}}+\||v|^{-\gamma}g_{1}\|_{\calm_{TV}} + \|f\|_{\LR{1}_{t,x,v}\cap \LR{\infty}_{t,x,v}} + \sup_{t\in[0,T]}\|\uf(t)\|_{\LR{1}_{x}}.
 \end{align*}
 We note that since trivially $f^<\in \LR{\infty}_{t,x,v}$ with norm bounded by $1$ we have by Theorem \ref{thm:wp-kinetic}
 \begin{align*}
  \norm{f}_{\LR{1}_{t,x,v} \cap \LR{\infty}_{t,x,v}} + \sup_{t\in[0,T]}\|\uf(t)\|_{\LR{1}_{x}}\lesssim  \norm{u}_{\LR{1}_{t,x}}+ 1 + \sup_{t\in[0,T]}\|\uf(t)\|_{\LR{1}_{x}}\lesssim \norm{u_0}_{\LR{1}_x}+ \|S\|_{L^1_{t,x}}+ 1.
 \end{align*}
Next, we check that $|v|^{1-\gamma}g_0 \in \calm_{TV}$. Indeed, since $|v|^{1-\gamma}$ can be estimated by a constant on the supports of $\psi_0$ and $\partial_v\psi_0$, we may apply Lemma \ref{lem:ph-est-1} to the effect of
\begin{align*}
 \||v|^{1-\gamma}g_0\|_{\calm_{TV}}&=\||v|^{1-\gamma}(\vp\psi_0\delta_{v=u(t,x)}S + \partial_t\vp f^< - \vp q\partial_\vf\psi_0)\|_{\calm_{TV}} \\
 &\lesssim \|S\|_{L^1_{t,x}} + \|\partial_t\vp |u|\|_{L^1_{t,x}}+\|q\partial_\vf\psi_0\|_{\calm_{TV}}\\
 &\lesssim \|\partial_t\vp |u|\|_{L^1_{t,x}}+\|u_0\|_{L^1_x} + \|S\|_{L^1_{t,x}}.
\end{align*}
Utilizing Lemma \ref{lem:ph-est-1} once more to the effect of
\begin{align*}
 \||v|^{-\gamma}g_1\|_{\calm_{TV}}\lesssim \||v|^{-\gamma}q^<\|_{\calm_{TV}}\lesssim \|u_0\|_{L^{1}_{x}}+ \|S\|_{L^1_{t,x}},
\end{align*}
we obtain by Sobolev embedding
\begin{align}\label{lem:pme_Lmbound_l1_11}
  \|\vp u^<\|_{\LR{m}_{t}\LR{\rho}_x}\lesssim \norm{\vp\uf^<}_{L^{m}_{t}\WSR{\sigma_x}{m}_x}\lesssim \|u_0\|_{L^1_x} + \|\partial_t\vp |u|\|_{L^1_{t,x}}+ \|S\|_{L^1_{t,x}}+1.
 \end{align}
With the same construction $\vp_n \to 1_{[0,T]}$ as in the proof of Theorem \ref{lem:pme}, this gives \eqref{cor:pme_l1_s1}.
 
 \refstepcounter{cor:pme_l1_prf} 
  \textit{Step} \arabic{cor:pme_l1_prf}\label{cor:pme_l1_prf_st2}. Next, we investigate $u^>$ and establish for $\eta\in(1,m)$ and $\eta^*=\frac{\eta d(m-1)}{d(m-1)-2(\eta-1)}$ the bound
  \begin{align}\label{cor:pme_l1_s2}
  \|u^>\|_{\LR{\eta}_{t}\LR{\eta^*}_x}\lesssim \|u_0\|_{L^1_x} + \|S\|_{L^1_{t,x}} +1.
 \end{align}
  Set $g_0:=\vp\psi_1\delta_{v=u(t,x)}S + \partial_t\vp f^> + \vp q\partial_\vf\psi_0$ and $g_1:=\vp q^>$. 
  Choose $\gamma\in(1,m)$ sufficiently small, so that $\eta\in(1,m+1-\gamma)$, and define
\begin{align*}
 \sigma_x:=\frac{\eta-1}{\eta}\frac{2}{m-1}\in\vpp{0,\frac{\eta-1}{\eta}\frac{2}{m-\gamma}}.
\end{align*}
 We apply Corollary \ref{cor:av0} with $\mu=1$ and $q=\eta$, which gives
\begin{align*}
  \lefteqn{\norm{\vp\uf^>}_{L^{\eta}_{t}\WSR{\sigma_x}{\eta}_x}} \\
   &&\quad \lesssim \||v|^{1-\gamma}g_{0}\|_{\calm_{TV}}+\||v|^{-\gamma}g_{1}\|_{\calm_{TV}} + \|\vp f^>\|_{\LR{1}_{t,x,v}\cap \LR{\infty}_{t,x,v}} + \|\vp \uf^>\|_{\LR{1}_{t,x}\cap \LR{\eta}_{t}\LR{1}_{x}} \\
  &&\lesssim \||v|^{1-\gamma}g_{0}\|_{\calm_{TV}}+\||v|^{-\gamma}g_{1}\|_{\calm_{TV}} + \|f\|_{\LR{1}_{t,x,v}\cap \LR{\infty}_{t,x,v}} + \sup_{t\in[0,T]}\|\uf(t)\|_{\LR{1}_{x}}.
 \end{align*}
 The terms involving $f$ and $u$ are estimated as in Step \ref{cor:pme_l1_prf_st1}. Further, since $|v|^{1-\gamma}$ can be estimated by a constant on the support of $\psi_1$ and $\partial_v\psi_0$, we have by Lemma \ref{lem:ph-est-1}
 \begin{align*}
\||v|^{1-\gamma}g_0\|_{\calm_{TV}}&=\||v|^{1-\gamma}(\vp\psi_1\delta_{v=u(t,x)}S + \partial_t\vp f^> + \vp q\partial_\vf\psi_0)\|_{\calm_{TV}} \\
&\lesssim \|S\|_{L^1_{t,x}} +\|\partial_t\vp |u|\|_{L^1_{t,x}}+\|q\partial_\vf\psi_0\|_{\calm_{TV}}\\
 &\lesssim \|\partial_t\vp |u|\|_{L^1_{t,x}}+\|u_0\|_{L^1_x}+ \|S\|_{L^1_{t,x}},
\end{align*}
and, again due to Lemma \ref{lem:ph-est-1},
\begin{align*}
 \||v|^{-\gamma}g_1\|_{\calm_{TV}}\lesssim \||v|^{-\gamma}q^>\|_{\calm_{TV}}\lesssim \|u_0\|_{L^{1}_{x}} + \|S\|_{L^1_{t,x}}.
\end{align*}
Since $\eta^*=\frac{\eta d}{d-\sigma_x\eta}$, we have by Sobolev embedding $\WSR{\sigma_x}{\eta}_x\subset L^{\eta^*}_x$, and hence
\begin{align*}
  \|\vp u^>\|_{\LR{\eta}_{t}\LR{\eta^*}_x}\lesssim \|\vp u^>\|_{L^{\eta}_{t}\WSR{\sigma_x}{\eta}_x}\lesssim \|u_0\|_{L^1_x} + \|\partial_t\vp |u|\|_{L^1_{t,x}}+ \|S\|_{L^1_{t,x}}+1.
 \end{align*}
With the same construction $\vp_n \to 1_{[0,T]}$ as before, this yields \eqref{cor:pme_l1_s2}.

 \refstepcounter{cor:pme_l1_prf} 
  \textit{Step} \arabic{cor:pme_l1_prf}\label{cor:pme_l1_prf_st3}. In this step, we show that for $s\in [1,m+\frac{2}{d})$ we have
  \begin{align}\label{lem:pme_Lmbound_l1}
  \|u\|_{\LR{s}_{t,x}}\lesssim  \|u_0\|_{L^1_x} + \|S\|_{L^1_{t,x}}+1.
 \end{align}
  Observe that it suffices to show the assertion for $s>m$, since $u\in \LR{1}(0,T;\LR{1}(\R^d))$ is already established by Theorem \ref{thm:wp-kinetic}.
  
  Define $\rho:=\frac{m}{m+1-s}\in(m,\frac{md}{d-2})$. For $\vt\in(0,1)$, it holds $[\LR{\infty}_t\LR{1}_x,\LR{m}_t\LR{\rho}_x]_\vt=\LR{p_\vt}_t\LR{q_\vt}_x$ with
  \begin{align*}
   \frac{1}{p_\vt}=\frac{\vt}{m} \qquad \tand \quad \frac{1}{q_\vt}=1-\vt+\frac{\vt}{\rho}.
  \end{align*}
  Choosing $\vt:=\frac{m\rho}{m\rho +\rho-m}\in(0,1)$, we obtain $p_\vt=q_\vt=s$, and hence by \eqref{cor:pme_l1_s1} and Theorem \ref{thm:wp-kinetic}
  \begin{align}\label{lem:pme_Lmbound_l1_1}
   \|u^<\|_{\LR{s}_{t,x}}\lesssim \|u^<\|_{\LR{\infty}_{t}\LR{1}_x}+\|u^<\|_{\LR{m}_{t}\LR{\rho}_x}\lesssim  \|u_0\|_{L^1_x} + 1.
  \end{align}
  Next, we define
  \begin{align*}
   \eta:=\frac{sd(m-1)+2}{d(m-1)+2}\in(1,m) \qquad \tand \quad \eta^*=\frac{\eta d}{d-2\frac{\eta-1}{m-1}}
  \end{align*}
  and observe that for $\vt\in(0,1)$, it holds $[\LR{\infty}_t\LR{1}_x,\LR{\eta}_t\LR{\eta^*}_x]_\vt = \LR{p_\vt}_t\LR{q_\vt}_x$ with
  \begin{align*}
   \frac{1}{p_\vt}=\frac{\vt}{\eta} \qquad \tand \quad \frac{1}{q_\vt}=1-\vt+\frac{\vt}{\eta^*}.
  \end{align*}
  Choosing $\vt:=\frac{\eta d(m-1)}{\eta d(m-1) + 2(\eta-1)}\in(0,1)$, we obtain $p_\vt=q_\vt=s$, and hence by \eqref{cor:pme_l1_s2} and Theorem \ref{thm:wp-kinetic}
  \begin{align}\label{lem:pme_Lmbound_l1_2}
  \|u^>\|_{\LR{s}_{t,x}}\lesssim \|u^>\|_{\LR{\infty}_{t}\LR{1}_x}+\|u^>\|_{\LR{\eta}_{t}\LR{\eta^*}_x}\lesssim \|u_0\|_{L^1_x} + \|S\|_{L^1_{t,x}}+1.
 \end{align}
  Combining \eqref{lem:pme_Lmbound_l1_1} and \eqref{lem:pme_Lmbound_l1_2}, we obtain \eqref{lem:pme_Lmbound_l1}.

  \refstepcounter{cor:pme_l1_prf} 
  \textit{Step} \arabic{cor:pme_l1_prf}\label{cor:pme_l1_prf_st4}. In this step we argue that
  \begin{align*}
  \norm{\vp u^<}_{\WSR{\sigma_t}{p}(\WSR{\sigma_x}{p})}&\lesssim \|\partial_t\vp |u|\|_{L^1_{t,x}} + \|u_0\|_{L^1_x}^m + \|S\|_{L^1_{t,x}}^m+1.
 \end{align*}
  Indeed, we choose $\gamma\in(0,1)$ so large that $\sigma_x<\frac{p-2+\gamma}{p}\frac{2}{m-1}$ and $m+1-\gamma<m+\frac{2}{d}$. Then we apply Corollary \ref{cor:av3} with $g_0:=\vp\psi_0\delta_{v=u(t,x)}S + \partial_t\vp f^< - \vp q \partial_v \psi_0$, $g_1:=\vp q^<$ and $\td p=p$. We obtain by \eqref{cor:av3_main_est} some $r\in (p,m+1-\gamma)$ such that
  \begin{align*}
  \norm{\vp u^<}_{\WSR{\sigma_t}{p}(\WSR{\sigma_x}{p})}&\lesssim \|g_{0}\|_{\calm_{TV}} + \||v|^{1-\gamma}g_{0}\|_{\calm_{TV}}+\||v|^{-\gamma}g_{1}\|_{\calm_{TV}} + \|f\|_{\LR{1}_{t,x,v}\cap \LR{\infty}_{t,x,v}} \\
  &\quad + \|u\|_{\LR{1}_t\LR{p}_x\cap \LR{r}_{t,x}} + \||u|^m\|_{\LR{1}_{t,x}}.
 \end{align*}
  The first four terms on the right-hand side can be estimated as in Step \ref{cor:pme_l1_prf_st1} (indeed, we did not use the coefficient $|v|^{1-\gamma}$ in the estimate of $g_0$) via
  \begin{align*}
   \|g_{0}\|_{\calm_{TV}} + \||v|^{1-\gamma}g_{0}\|_{\calm_{TV}}&+\||v|^{-\gamma}g_{1}\|_{\calm_{TV}} + \|f\|_{\LR{1}_{t,x,v}\cap \LR{\infty}_{t,x,v}} \\
   &\lesssim \|\partial_t\vp |u|\|_{L^1_{t,x}} + \|u_0\|_{L^1_x} + \|S\|_{L^1_{t,x}}+1,
  \end{align*}
  while the last two terms are estimated in virtue of $r<m+1-\gamma<m+\frac{2}{d}$ through \eqref{lem:pme_Lmbound_l1} as
  \begin{align*}
   \|u\|_{\LR{1}_t\LR{p}_x\cap \LR{r}_{t,x}} + \||u|^m\|_{\LR{1}_{t,x}}\lesssim \|u\|_{\LR{p}_{t,x}\cap \LR{r}_{t,x}} + \|u\|_{\LR{m}_{t,x}}^m\lesssim \|u_0\|_{L^1_x}^m + \|S\|_{L^1_{t,x}}^m+1.
  \end{align*}
  
  \refstepcounter{cor:pme_l1_prf} 
  \textit{Step} \arabic{cor:pme_l1_prf}\label{cor:pme_l1_prf_st5}.
  In this step we establish
  \begin{align}\label{cor:pme_l1_prf_s5}
  \norm{\vp u^>}_{\WSR{\sigma_t}{p}(\WSR{\sigma_x}{p})}\lesssim \|\partial_t\vp |u|\|_{L^1_{t,x}} + \|u_0\|_{L^1_x}^m + \|S\|_{L^1_{t,x}}^m +1.
  \end{align}
  Assume first $p<m$. Choose $\gamma\in(1,m)$ so small that $p\in(1,m+1-\gamma)$ and $\sigma_t<\frac{m+1-\gamma-p}{p}\frac{1}{m-1}$ and apply Corollary \ref{cor:av3} with $g_0:=\vp\psi_1\delta_{v=u(t,x)}S + \partial_t\vp f^> + \vp q \partial_v \psi_0$, $g_1:=\vp q^>$ and $\td p=p$. Estimate \eqref{cor:av3_main_est} gives
  \begin{align*}
  \norm{\vp u^>}_{\WSR{\sigma_t}{p}(\WSR{\sigma_x}{p})}&\lesssim \|g_{0}\|_{\calm_{TV}} + \||v|^{1-\gamma}g_{0}\|_{\calm_{TV}}+\||v|^{-\gamma}g_{1}\|_{\calm_{TV}} \\
  & \quad + \|f\|_{\LR{1}_{t,x,v}\cap \LR{\infty}_{t,x,v}} 
   +  \|u\|_{\LR{1}_t\LR{p}_x\cap \LR{r}_{t,x}} + \||u|^m\|_{\LR{1}_{t,x}}.
 \end{align*}
  The first four terms on the right-hand side are estimated as in Step \ref{cor:pme_l1_prf_st2} via
  \begin{align*}
   \|g_{0}\|_{\calm_{TV}} + \||v|^{1-\gamma}g_{0}\|_{\calm_{TV}} &+\||v|^{-\gamma}g_{1}\|_{\calm_{TV}} + \|f\|_{\LR{1}_{t,x,v}\cap \LR{\infty}_{t,x,v}} \\
   &\lesssim \|\partial_t\vp |u|\|_{L^1_{t,x}} + \|u_0\|_{L^1_x} + \|S\|_{L^1_{t,x}}+1,
  \end{align*}
  while the last two terms are estimated through \eqref{lem:pme_Lmbound_l1} as
  \begin{align*}
   \|u\|_{\LR{1}_t\LR{p}_x\cap \LR{r}_{t,x}} + \||u|^m\|_{\LR{1}_{t,x}}\lesssim \|u\|_{\LR{p}_{t,x}\cap \LR{r}_{t,x}} + \|u\|_{\LR{m}_{t,x}}^m\lesssim \|u_0\|_{L^1_x}^m + \|S\|_{L^1_{t,x}}^m+1.
  \end{align*}
  Hence, we have shown \eqref{cor:pme_l1_prf_s5} in the case $p\in(1,m)$. If $p=m$, we choose $p_0\in(1,m)$ sufficiently large such that for $\kappa_x(p_0):=\frac{p_0-1}{p_0}\frac{2}{m-1}$ it holds $\kappa_x(p_0)-\frac{d}{p_0}>\sigma_x-\frac{d}{m}$. We observe that for $\kappa_t(p_0):=\frac{m-p_0}{p_0}\frac{1}{m-1}$ it holds $\kappa_t(p_0)-\frac{1}{p_0}>\sigma_t-\frac{1}{m}$ due to $p_0<m$ (indeed, we have necessarily $\sigma_t=0$). Choosing sufficiently large $\sigma_x(p_0)<\kappa_x(p_0)$ and $\sigma_t(p_0)<\kappa_t(p_0)$, we conclude by Sobolev embedding
  \begin{align*}
   \norm{\vp u^>}_{\LR{m}_t(\WSR{\sigma_x}{m}_x)}&\lesssim \norm{\vp u^>}_{\WSR{\sigma_t(p_0)}{p_0}(\WSR{\sigma_x(p_0)}{p_0})}\\
   &\lesssim \|\partial_t\vp |u|\|_{L^1_{t,x}} + \|u_0\|_{L^1_x}^m + \|S\|_{L^1_{t,x}}^m +1,
  \end{align*}
  which is \eqref{cor:pme_l1_prf_s5} in the case $p=m$.
  
  \refstepcounter{cor:pme_l1_prf} 
  \textit{Step} \arabic{cor:pme_l1_prf}\label{cor:pme_l1_prf_st6}. Conclusion.
  With the same construction $\vp_n \to 1_{[0,T]}$ as in the proof of Theorem \ref{lem:pme}, Steps \ref{cor:pme_l1_prf_st4} and \ref{cor:pme_l1_prf_st5} combine to
  \begin{align*}
   \sup_{n\in\N}\norm{\vp_n u}_{\WSR{\sigma_t}{p}(\WSR{\sigma_x}{p})}&\lesssim \sup_{n\in\N}\norm{\vp_n u^<}_{\WSR{\sigma_t}{p}(\WSR{\sigma_x}{p})} + \sup_{n\in\N} \norm{\vp_n u^>}_{\WSR{\sigma_t}{p}(\WSR{\sigma_x}{p})} \\
   &\lesssim \|u_0\|_{L^1_x}^m + \|S\|_{L^1_{t,x}}^m +1.
  \end{align*}
  Since $\vp_n u\to u 1_{[0,T]}$ in the sense of distributions, we obtain \eqref{cor:pme_est1_l1} by the weak lower semi-continuity of the norm in $\WSR{\sigma_t}{p}(0,T;\WSR{\sigma_x}{p}(\R^d))$. 
  Estimate \eqref{cor:pme_est2_l1} follows analogously to the proof of Corollary \ref{cor:pme} (\ref{cor:pme_2}). \qedhere
\end{proof}

\appendix

\section{\label{app:kin_solutions}Kinetic Solutions}

In this section we recall some details on the concept of entropy / kinetic solutions and their well-posedness for partial differential equations of the type
\begin{align}
\partial_{t}u+\div A(u) & =\div(b(u)\nabla u)+S(t,x)\quad\text{on }(0,T)\times\R_{x}^{d}\label{eq:par-hyp}\\
u(0) & =u_{0}\quad\text{on }\R_{x}^{d},\nonumber 
\end{align}
where
\begin{align}
u_{0} & \in L^{1}(\R_{x}^{d}),\,S\in L^{1}([0,T]\times\R_{x}^{d}),\,T\ge0,\nonumber \\
a:=A' & \in C(\R;\R^{d})\cap C^{1}(\R\setminus\{0\};\R^{d}),\label{eq:ph-as}\\
b=(b_{jk})_{j,k=1\dots d} & \in C(\R;S_{+}^{d\times d})\cap C^{1}(\R\setminus\{0\};S_{+}^{d\times d}).\nonumber 
\end{align}
Here, $S_{+}^{d\times d}$ denote the space of symmetric, non-negative definite matrices. For $b=(b)_{i,j=1\dots d}\in S_{+}^{d\times d}$ we set $\s=b^{\frac{1}{2}}$, that is, $b_{i,j}=\sum_{k=1}^{d}\sigma_{i,k}\s_{k,j}$. For a locally bounded function $b:\R\to S_{+}^{d\times d}$ we let $\b_{i,k}$ be such that $\b_{i,k}'(v)=\sigma_{i,k}(v)$. Similarly, for $\psi\in C_{c}^{\infty}(\R_{v})$ we let $\b_{i,j}^{\psi}$ be such that $(\b_{i,k}^{\psi})'(v)=\psi(v)\sigma_{i,k}(v)$.
The corresponding kinetic form of \eqref{eq:par-hyp} reads (cf.~\cite{ChP03})  
\begin{align}
\call(\partial_{t},\nabla_{x},v)f(t,x,v) & =\partial_{t}f+a(v)\cdot\nabla_{x}f-\div(b(v)\nabla_{x}f)\label{eq:kinetic_ph-3}\\
 & =\partial_{v}q+S(t,x)\d_{u(t,x)=v}(v),\nonumber 
\end{align}
where $q\in\calm^{+}$ and $\call$ is identified with the symbol 
\begin{equation}
\call(i\tau,i\xi,v):=i\tau+a(v)\cdot i\xi-(b(v)\xi,\xi).\label{eq:ph-symbol}
\end{equation}
We will use the terms kinetic and entropy solution synonymously. From \cite{ChP03} we recall the definition of entropy/kinetic solutions to \eqref{eq:par-hyp}.
\begin{defn}
\label{def:kinetic_sol-1}We say that $u\in C([0,T];L^{1}(\R^{d}))$ is an entropy solution to \eqref{eq:par-hyp} if the corresponding kinetic function $f$ satisfies

\begin{enumerate}
\item For any non-negative $\psi\in\cald(\R)$, $k=1,\dots,d$,
\[
\sum_{i=1}^{d}\partial_{x_{i}}\b_{ik}^{\psi}(u)\in L^{2}([0,T]\times\R^{d}).
\]
\item For any two non-negative functions $\psi_{1},\psi_{2}\in\cald(\R)$, $k=1,\dots,d$,
\[
\sqrt{\psi_{1}(u(t,x))}\sum_{i=1}^{d}\partial_{x_{i}}\b_{ik}^{\psi_{2}}(u(t,x))=\sum_{i=1}^{d}\partial_{x_{i}}\b_{ik}^{\psi_{1}\psi_{2}}(u(t,x))\quad\text{a.e.}.
\]
\item There are non-negative measures $m,n\in\calm^{+}$ such that, in the sense of distributions,
\[
\partial_{t}f+a(v)\cdot\nabla_{x}f-\div(b(v)\nabla_{x}f)=\partial_{v}(m+n)+\delta_{v=u(t,x)}S\quad\text{on }(0,T)\times\R_{x}^{d}\times\R_{v}
\]
where $n$ is defined by
\[
\int\psi(v)n(t,x,v)\dd v=\sum_{k=1}^{d}\left(\sum_{i=1}^{d}\partial_{x_{i}}\b_{ik}^{\psi}(u(t,x))\right)^{2}
\]
for any $\psi\in\cald(\R)$ with $\psi\ge0.$
\item We have
\[
\int(m+n)\dd x \dd t\le\mu(v)\in L_{0}^{\infty}(\R),
\]
where $L_{0}^{\infty}$ is the space of $L^{\infty}$-functions vanishing for $|v|\to\infty$.
\end{enumerate}
\end{defn}

The well-posedness of entropy solutions to \eqref{eq:par-hyp} follows along the same lines of \cite{ChP03}. In this form, it can be found in \cite{Ges17}.
\begin{thm}
\label{thm:wp-kinetic}Let $u_{0}\in L^{1}(\R^{d})$, $S\in L^{1}([0,T]\times\R^{d}).$ Then there is a unique entropy solution $u$ to \eqref{eq:par-hyp} satisfying $u\in C([0,T];L^{1}(\R^{d}))$. For two entropy solutions $u^{1}$, $u^{2}$ with initial conditions $u_{0}^{1},u_{0}^{2}$ and forcing $S^{1},S^{2}$ we have 
\[
\sup_{t\in[0,T]}\|u^{1}(t)-u^{2}(t)\|_{L^{1}(\R^{d})}\le\|u_{0}^{1}-u_{0}^{2}\|_{L^{1}(\R^{d})}+\|S^{1}-S^{2}\|_{L^{1}([0,T]\times\R^{d})}.
\]
\end{thm}
Furthermore, the following \emph{a priori} estimate was given in Lemma 2.3 in \cite{Ges17}.
\begin{lem}
\label{lem:ph-est}Let $u$ be the unique entropy solution to \eqref{eq:par-hyp} with $u_{0}\in(L^{1}\cap L^{2-\gamma})(\R_{x}^{d})$, $S\in(L^{1}\cap L^{2-\gamma})([0,T]\times\R_{x}^{d})$ for some $\gamma\in(-\infty,1)$. Then, there is a constant $C=C(T,\g)\ge0$ such that
\[
 \sup_{t\in[0,T]}\|u(t)\|_{L_{x}^{2-\gamma}}^{2-\gamma}+(1-\gamma)\int_{0}^{T}\int_{\R^{d+1}}|v|^{-\gamma}q\,\dd v \dd x \dd r
 \le C\big(\|u_{0}\|_{L_{x}^{2-\gamma}}^{2-\gamma}+\|S\|_{L_{t,x}^{2-\gamma}}^{2-\gamma}\big).
\]
\end{lem}

In the case of $L^{1}$ initial data a different proof for the existence of singular moments of the kinetic measure $q$ is needed.

\begin{lem} \label{lem:ph-est-1}Let $u$ be the unique entropy solution to \eqref{eq:par-hyp} with $u_{0}\in L^{1}(\R_{x}^{d})$, $S\in L^{1}([0,T]\times\R_{x}^{d})$. Then, the map $v\to\int_{0}^{T}\int_{\R^{d}_x}q(r,x,v)\,\dd x\dd r$ is continuous and, for all $v_{0}\in\R_{v}$, we have
\begin{equation}
\begin{split}
\int_{0}^{T}\int_{\R_{x}^{d}}q(r,x,v_{0})\,\dd x\dd r
&\le\int_{\R_{x}^{d}}(\sgn(v_{0})(u_{0}-v_{0}))_{+}\,\dd x \\
& \ \ \ +\int_{0}^{T}\int_{\R_{x}^{d}}\sgn_{+}(\sgn(v_{0})(u-v_{0}))S\,\dd x\dd r\\
&\le\int_{\R_{x}^{d}}|u_{0}|\,\dd x+\int_{0}^{T}\int_{\R_{x}^{d}}|S|\,\dd x\dd r.
\end{split}
\end{equation}
\end{lem}\begin{proof}After a standard cut-off argument, the kinetic formulation yields, for every $\eta\in C_{c}^{\infty}(\R_{v})$,
\begin{equation}
\int_{0}^{T}\int_{\R^{d+1}}\partial_{v}\eta q \,\dd v\dd x\dd r=-\int_{\R_{v}}\eta\left(\int_{\R_{x}^{d}}f\,\dd x|_{0}^{T}\right)\dd v+\int_{0}^{T}\int_{\R_{x}^{d}}\eta(u)S \,\dd x\dd r.\label{eq:ctn}
\end{equation}
This first implies that $v\to\int_{0}^{T}\int_{\R^{d}_x}q(r,x,v)\,\dd x\dd r$ has left and right limits, which, again due to \eqref{eq:ctn} have to coincide. Let now $v_{0}\in\R_{+}$. The claim then follows by choosing $\psi^{n}$ to be smooth, non-negative, convex approximations of $(v-v_{0})_{+}$, and $\eta^{n}:=(\psi^{n})'$, which, taking the limit $n\to\infty$, yields
\begin{align*}
\int_{0}^{T}\int_{\R^{d}_x}q(t,x,v_{0})\,\dd x\dd r & =-\int_{\R_{x}^{d}}(u-v_{0})_{+}\,\dd x|_{0}^{T}+\int_{0}^{T}\int_{\R_{x}^{d}}\sgn_{+}(u-v_{0})S\,\dd x\dd r\\
 & \le \int_{\R_{x}^{d}}(u_{0}-v_{0})_{+}\,\dd x+\int_{0}^{T}\int_{\R_{x}^{d}}\sgn_{+}(u-v_{0})S\,\dd x\dd r.
\end{align*}
The case $v_{0}\in\R_{-}$ is treated analogously replacing $(v-v_{0})_{+}$ by $(v-v_{0})_{-}.$
\end{proof}

\section{\label{app:mult}Fourier Multipliers}

In this section, we provide some Fourier multiplier results well-adapted to our Averaging Lemma \ref{lem:av}. We recall the definition of $\dot\R^{d+1}$ and of the functions $\eta_l$ and $\vp_j$ given in Section \ref{FctSpc}, and define $\td\eta_l:=\eta_{l-1}+\eta_l+\eta_{l+1}$ and $\td\vp_j:=\vp_{j-1}+\vp_{j}+\vp_{j+1}$. We observe $\td\eta_l(2^l\cdot)=\td\eta_0$ and $\td\vp(2^j\cdot)=\td\vp_0$. Moreover, $\td\eta_l$ and $\td\vp_j$ are identically unity on the support of $\eta_l$ and $\vp_j$, respectively.

\begin{thm}\label{thm:FM}
 Let $k=2+2[1+d/2]$. Let $m:\dot\R^{d+1}\to \C$ be $k$-times differentiable and such that for all $\alpha=(\alpha_\tau,\alpha_\xi)\in\N_0\times\N_0^{d}$ with $|\alpha|\le k$ there is a constant $C_\alpha$ such that for all $(\tau,\xi)\in \dot\R^{d+1}$
 \begin{align}\label{FM_assump}
  |\partial_\tau^{\alpha_\tau}\partial_\xi^{\alpha_\xi} m(\tau,\xi)|\le C_\alpha|\tau|^{-\alpha_\tau}|\xi|^{-|\alpha_\xi|}.
 \end{align}
 Then there is a constant $C>0$, depending only on the constants $C_\alpha$, such that for any $p\in[1,\infty]$ and all $l,j\in\Z$, we have
 \begin{align}\label{FM_est_1}
  \norm{\td\eta_l\td\vp_jm}_{\calm^{p}}\le C,
 \end{align}
 i.e., $\td\eta_l\td\vp_jm$ (more precisely the mapping $(\tau,\xi)\mapsto \td\eta_l(\tau)\td\vp_j(\xi)m(\tau,\xi)$) extends to an $\LR{p}_{t,x}$-multiplier with a norm independent of $l$ and $j$. Furthermore, this mapping extends to an $\calm_{TV}$-multiplier with the same norm bound.
\end{thm}
\begin{proof}
Since $\norm{\cdot}_{\calm^p}\le \norm{\cdot}_{\calm^1}$, it suffices to estimate the $L^1$ multiplier norm of $\td\eta_l\td\vp_j m$ in order to obtain \eqref{FM_est_1}. Since multiplier norms are invariant under dilation and since $\|m\|_{\calm^1}$ is equal to the total mass of $\calf^{-1}m$, see \cite[Theorem 6.1.2]{BeL76}, we have
 \begin{align*}
  \norm{\td\eta_l\td\vp_j m}_{\calm^1}=\norm{\td\eta_0\td\vp_0 m_{l,j}}_{\calm^1}=\norm{\calf^{-1}_{t,x}\td\eta_0\td\vp_0 m_{l,j}}_{\LR{1}_{t,x}}, 
 \end{align*}
 where $m_{l,j}(\tau,\xi):=m(2^l\tau,2^j\xi)$.
Let $M:=[1+d/2]$. We observe
 \begin{align*}
  (&1+t^2)(1+|x|^2)^M\calf^{-1}_{t,x}[\td\eta_0\td\vp_0 m_{l,j}](t,x) \\
  &=c_d\int_{\R_t\times\R^d_x} (\id-\partial_\tau^2)(\id-\Delta_\xi)^M\left(e^{it\tau+ix\cdot\xi}\right)\td\eta_0(\tau)\td\vp_0(\xi)m(2^l\tau,2^j\xi)\dd\xi\dd\tau \\
  &=c_d\int_{\R_t\times\R^d_x} e^{it\tau+ix\cdot\xi} (\id-\partial_\tau^2)(\id-\Delta_\xi)^M\left(\td\eta_0(\tau)\td\vp_0(\xi)m(2^l\tau,2^j\xi)\right)\dd\xi\dd\tau \\
  &=\hspace*{-.6cm} \sum_{\stackrel{\alpha_\tau+\beta_\tau\le 2}{|\alpha_\xi|+|\beta_\xi|\le 2M}}
  \hspace*{-.6cm}c_{d,\alpha,\beta}2^{l\beta_\tau}2^{j|\beta_\xi|} 
  \int_{\R_t\times\R^d_x}  \hspace*{-.6cm} e^{it\tau+ix\cdot\xi} \partial_\tau^{\alpha_\tau}\td\eta_0(\tau)\partial_{\xi}^{\alpha_\xi}\td\vp_0(\xi)\partial_\tau^{\beta_\tau}\partial_\xi^{\beta_\xi}m(2^l\tau,2^j\xi)\dd\xi\dd\tau,
 \end{align*}
 where $c_d$ and $c_{d,\alpha,\beta}$ are constants that do not depend on $l$ and $j$. On $\supp \td\eta_0\times \supp \td\vp_0$ we have $|\partial_\tau^{\beta_\tau}\partial_\xi^{\beta_\xi}m(2^l\tau,2^j\xi)|\le C_\beta 2^{-l\beta_\tau}2^{-j|\beta_\xi|}$, and hence we obtain
 \begin{align*}
  \displaystyle (1+t^2)(1+|x|^2)^M|\calf^{-1}_{t,x}[\td\eta_0\td\vp_0 m_{l,j}](t,x)|\le c.
 \end{align*}
 Since $2M>d$, it follows $\norm{\calf^{-1}_{t,x}[\td\eta_0\td\vp_0 m_{l,j}]}_{\LR{1}_{t,x}}\le C$, which yields \eqref{FM_est_1}. In particular, $\td\eta_l\td\vp_j m$ is an $L^1$-multiplier with a norm bound independent of $l$ and $j$, and as such extends to a multiplier on $\calm_{TV}$ with the same norm bound.
\end{proof}

\begin{rem}
 In Theorem \ref{thm:FM}, the assumptions on the differentiability of $m$ may be relaxed: Indeed, the proof shows that it suffices to assume that $m$ is a continuous function such that $\partial_\tau^{\alpha_\tau}m$, $\partial_\xi^{\alpha_x}m$ and $\partial_\tau^{\alpha_\tau}\partial_\xi^{\alpha_x}m$ exist for all $\alpha=(\alpha_\tau,\alpha_\xi)$ with $\alpha_\tau\le 2$ and $|\alpha_\xi|\le 2[1+d/2]$, and that \eqref{FM_assump} holds for these choices of $\alpha$.
\end{rem}

\begin{rem}\label{rem:FM}
 Clearly, Theorem \ref{thm:FM} has an isotropic variant, cf.\@ \cite[Lemma 2.2]{BCD11}. More precisely, a simple adaptation of the proof shows the following: Let $k=2[1+d/2]$. Let $m:\R^{d}\setminus\set{0}\to \C$ be $k$-times differentiable and such that for all $\alpha\in\N_0^{d}$ with $|\alpha|\le k$ there is a constant $C_\alpha$ such that for all $\xi\in \R^{d}\setminus\set{0}$ we have $|\partial^{\alpha} m(\xi)|\le C_\alpha|\xi|^{-|\alpha|}$. Then there is a constant $C>0$, depending only on the constants $C_\alpha$, such that for any $p\in[1,\infty]$ and all $j\in\Z$, we have $\norm{\td\vp_jm}_{\calm^{p}}\le C$. Again, $\td\vp_jm$ extends to an $\calm_{TV}$-multiplier (in $\xi$) with the same norm bound.
\end{rem}

\begin{lem}\label{lem:FM}
 Let $\call$ be defined as in \eqref{av_op} and fix $\alpha=(\alpha_\tau,\alpha_\xi)\in\N_0\times\N_0^{d}$. Then we have for all $(\tau,\xi,v)\in\dot\R^{d+1}\times\R$ the estimate
 \begin{align*}
  \left|\partial_\tau^{\alpha_\tau}\partial_\xi^{\alpha_\xi} \frac{1}{\call(i\tau,i\xi,v)}\right|\lesssim \frac{1}{|\call(i\tau,i\xi,v)|}|\tau|^{-\alpha_\tau}|\xi|^{-|\alpha_\xi|}.
 \end{align*}
\end{lem}
\begin{proof}
 The proof rests on the identity
 \begin{align*}
  \partial_\xi^{\alpha_\xi} \frac{1}{\call(i\tau,i\xi,v)}=\sum_{\beta}c_{\beta}\frac{\xi^{\beta} |v|^{(m-1)N_{\beta}}}{\call(i\tau,i\xi,v)^{1+N_{\beta}}},
 \end{align*}
 where $c_{\beta}$ are constants, $N_\beta:=\frac{|\alpha_\xi|+|\beta|}{2}$, and the sum runs over those $\beta\in\N_0^d$ with $|\beta|\le|\alpha_\xi|$ such that $|\alpha_\xi|+|\beta|$ is even. The identity can be proven easily by induction on the order of $\alpha_\xi$. From this and $\partial_\tau\call(i\tau,i\xi,v)=i$, it immediately follows
 \begin{align*}
  \left|\partial_\tau^{\alpha_\tau}\partial_\xi^{\alpha_\xi} \frac{1}{\call(i\tau,i\xi,v)}\right|\lesssim \sum_{\beta}\left|\frac{\xi^{\beta} |v|^{(m-1)N_{\beta}}}{\call(i\tau,i\xi,v)^{1+\alpha_\tau+N_{\beta}}}\right|,
 \end{align*}
 which in view of
 \begin{align*}
  \displaystyle \frac{|\xi|^{|\beta|} |v|^{(m-1)N_{\beta}}}{|\call(i\tau,i\xi,v)|^{N_{\beta}}}\le \frac{|\xi|^{|\beta|} |v|^{(m-1)N_{\beta}}}{(|v|^{m-1}|\xi|^2)^{N_{\beta}}} = |\xi|^{-(2N_{\beta}-|\beta|)}=|\xi|^{-|\alpha_\xi|}
 \end{align*}
 and
 \begin{align*}
  \displaystyle \frac{1}{|\call(i\tau,i\xi,v)|^{\alpha_\tau}}\le |\tau|^{-\alpha_\tau}
 \end{align*}
 yields the assertion.
\end{proof}


\noindent
\thanks{\textbf{Acknowledgment.}
The first author acknowledges financial support by the the Max Planck Society through the Max Planck Research Group ``Stochastic partial differential equations'' and by the DFG through the CRC ``Taming uncertainty and profiting from randomness and low regularity in analysis, stochastics and their applications''. Research of the third author was supported in part by NSF grants DMS16-13911, RNMS11-07444 (KI-Net) and ONR grant N00014-1812465. The hospitality of Laboratoire Jacques-Louis Lions in Sorbonne University and its support through ERC grant 740623 under the EU Horizon 2020 
is gratefully acknowledged.}


\frenchspacing
\bibliographystyle{plain}
\bibliography{pme}
 
\end{document}